\documentclass[a4paper, 10pt]{article}
\usepackage{amssymb}

\usepackage{amsmath,amssymb,amsthm}
\usepackage[latin1]{inputenc}
\usepackage{version,tabularx,multicol}
\usepackage{graphicx,float,psfrag}
\usepackage{stmaryrd}
 \usepackage{color}
\usepackage[pdftex]{hyperref}
\usepackage{amsfonts}
\usepackage{mathrsfs}
\usepackage{geometry}
\usepackage{cite}
\usepackage[numbers,sort&compress]{natbib}

\allowdisplaybreaks[4]

\theoremstyle{plain}
\newtheorem{theorem}{Theorem}[section]

 \newtheorem{lemma}[theorem]{Lemma}

\newtheorem{remark}{Remark}[section]

\newtheorem{assumption}{Assumption}[section]

\newcommand{\ind}{{\bf 1}}

  \makeatletter
  \@addtoreset{equation}{section}
  \makeatother

 \def\beqlb{\begin{eqnarray}}\def\eeqlb{\end{eqnarray}}
 \def\beqnn{\begin{eqnarray*}}\def\eeqnn{\end{eqnarray*}}

\newcommand{\bcen}{\begin{center}}
\newcommand{\ecen}{\end{center}}
\newcommand{\bgeqn}{\begin{equation}}
\newcommand{\edeqn}{\end{equation}}

\begin{document}
\title{On large deviation probabilities for empirical distribution of branching random walks with heavy tails}
\author{Shuxiong Zhang}
\maketitle


\noindent\textit{Abstract:}
Given a branching random walk $(Z_n)_{n\geq0}$ on $\mathbb{R}$, let $Z_n(A)$ be the number of particles located  in interval $A$ at generation $n$. It is well known (e.g., \cite{biggins}) that under some mild conditions, $Z_n(\sqrt nA)/Z_n(\mathbb{R})$ converges a.s. to $\nu(A)$ as $n\rightarrow\infty$, where $\nu$ is the standard Gaussian measure. In this work, we investigate its large deviation probabilities under the condition that the step size or offspring law has heavy tail, i.e. the decay rate of
$$\mathbb{P}(Z_n(\sqrt nA)/Z_n(\mathbb{R})>p)$$
as $n\rightarrow\infty$, where $p\in(\nu(A),1)$. Our results complete those in \cite{ChenHe} and \cite{Louidor}.

\bigskip
\noindent Mathematics Subject Classifications (2020): 60F10, 60J80, 60G50.

\bigskip
\noindent\textit{Key words and phrases}: Branching random walk; Large deviation; Empirical distribution; Heavy tail.

\section{Introduction and Main results}
\subsection{Introduction}
 In this work, we consider a  branching random walk (BRW) model, which is governed by a probability distribution $\{p_k\}_{k\geq 0}$ on natural numbers (called the offspring distribution) and a real valued random variable $X$ (called the step size or displacement). This model is defined as follows. At time $0$, there is one particle located at the origin. The particle dies and produces offsprings according to the offspring distribution $\{p_k\}_{k\geq 0}$. Afterwards, the offspring particles move independently according to the law of $X$. This forms a process $Z_1$. For any point process $Z_n$, $n\geq 2$, we define it by the following iteration
$$Z_n=\sum_{x\in Z_{n-1}}\tilde{Z}_1^{x},$$
where $\tilde{Z}_1^{x}$ has the same distribution as $Z_1(\cdot-S_x)$ and $\{\tilde{Z}_1^{x}: x\in Z_{n-1}\}$ (conditioned on $Z_{n-1}$) are independent. Here and later, for a point process (also for a point measure) $\xi$, $x\in\xi$ means $x$ is an atom of $\xi$, and $S_x$ is the position of $x$ (i.e., $\xi=\sum_{x\in\xi}\delta_{S_x}$ ).\par

In the present work, we are interested in the large deviation probabilities of the corresponding empirical distribution, which is defined as
$$\bar{Z}_n(A):=\frac{Z_n(A)}{Z_n(\mathbb{R})},\text{~for~measurable~set}~A\subset\mathbb{R}.$$
According to Biggins \cite[Theorem 6]{biggins}: if $\mathbb{E}[X]=0,~\mathbb{E}[X^2]=1$ and $\mathbb{E}[Z_1(\mathbb{R})\log Z_1(\mathbb{R})]<\infty$, then for any Borel measurable set $A\subset\mathbb{R}$,
$$\bar{Z}_n(\sqrt n A)\rightarrow \nu(A),~\mathbb{P}-a.s.~\text{on~non-extinction},$$
where $\nu$ is the standard Gaussian measure. So, it's natural to study the decay rate of
$$\mathbb{P}(\bar{Z}_n(\sqrt nA)\geq p)$$
as $n\rightarrow\infty$, where $p\in(\nu(A),1)$. \par
 In fact, this question has been considered by Louidor and Perkins \cite{Louidor} under the assumption that $p_0=p_1=0$ and $\mathbb{P}(X=1)=\mathbb{P}(X=-1)$. Later, Chen and He \cite{ChenHe} further investigated the problem for unbounded displacements; see also \cite{louidor2017large}. However, they always assume that the offspring law and step size have exponential moment (or step size has stretched exponential moment). So, in this paper, we shall deal with the case that offspring law or step size has heavy tail (i.e., $\mathbb{E}[e^{\theta Z_1(\mathbb{R})}]=\mathbb{E}[e^{\theta X}]=+\infty$ for any $\theta>0$). We will see that the strategy to study this problem and the answers will be very different from theirs.\par

We also  mention here that, since the last few decades, the model BRW has been extensively studied due to its connection to many fields, such as Gaussian multiplicative chaos, random walk in random environment, random polymer, random algorithms and discrete Gaussian free field etc; see \cite{hushi}, \cite{liu98}, \cite{liu06}, \cite{bramsonding15} and \cite{levelsetzhan} references therein. One can refer to Shi \cite{zhan} for a more detailed overview. The large deviation probabilities (LDP) for BRW and branching Brownian motion (BBM) on real line have attracted many researcher's attention. For example: Hu \cite{yhu}, Gantert and H{\"o}felsauer \cite{GWlower} and Chen and He \cite{Helower} considered the LDP and the moderate deviation of BRW's maximum (for BBM's maximum, see Chauvin and Rouault \cite{chauvin} and Derrida and Shi \cite{derrida16,derrida17,derrida172}); \"{O}z \cite{Mehmet} studied the lower deviation of BBM's local mass. See also A\"{i}dekon, Hu and Shi \cite{levelsetzhan}  for the upper deviation of BBM's level sets. Some other related works include Rouault \cite{rouault}, Buraczewski and Ma\'slanka  \cite{Burma} and Bhattacharya \cite{bhattacharya}.\par
 The almost surely behaviour of $\bar{Z}_n(\sqrt nA)$ has been considered by many researchers, e.g., Biggins \cite{biggins}, Harris \cite{harris}, Asmussen and Kaplan \cite{asmussen} and Klebaner \cite{klebaner}. Moreover, the almost surely convergence rate of $\bar{Z}_n(\sqrt nA)-\nu(A)$ tends to zero has also been well studied recently: Chen \cite{Chenxia} considered the branching Wiener process; Gao and Liu \cite{Gao} generalized Chen's results to the BRW, but a kind of Cram\'er's condition is needed for the step size; Gr\"ubel and Kabluchko \cite{Rudolf} studied the case when the step size of BRW is lattice. These results show that $\sqrt n(\bar{Z}_n(\sqrt n A)-\nu(A))$ converges almost surely to a non-degenerate limit.

\subsection{Main Results}\label{mainresults}
Before giving our results, we first introduce some notations. Let $\mathcal{A}$ be the algebra generated by $\{(-\infty,x],x\in\mathbb{R}\}$. For a   non-empty set $A\in\mathcal{A}$ and $p\in\left(\nu(A),1\right)$, define
\[\begin{split}
& I_A(p)=\inf\left\{\left|x\right|:\nu\left(A-x\right)\geq p\right\},\\
& J_A(p)=\inf\left\{r:\sup_{x\in\mathbb{R}}\nu\left(\left(A-x\right)/\sqrt{1-r}\right)\geq p,~r\in\left[0,1\right)\right\}.
\end{split}\]

\noindent Let $|Z_n|:=Z_n(\mathbb{R})$, $m:=\mathbb{E}[|Z_1|]$ and $b:=\min\{k\geq 0: p_k>0\}\leq B:=\sup\{k\geq 0: p_k>0\}\leq+\infty$. Recall that $\{p_k\}_{k\geq 0}$ is the offspring law, and $X$ is the step size.
  In the sequel of this work, we always need the following assumptions.
\begin{assumption}\label{assume}\leavevmode \\
(i) $p_0=0$, $p_1<1;$\\
(ii) $X$ is symmetric and $\mathbb{E}[X^2]=1;$\\
(iii) $A$ is a no-empty set in $\mathcal{A}$.
\end{assumption}
\begin{remark} In assumption (i), $p_0=0$ is not essential, just for convenience. If not the case, we can condition the probability on non-extinction to obtain analogous results. Assumption (ii) is not essential either, but simplifies the proof. Assumption (iii) is crucial to our main results, and if $A\notin\mathcal{A}$, then the situation would be very different (see \cite[Proposition 1.3]{Louidor}).
\end{remark}\par

Now we are ready to state our main results. The first theorem concerns the case that the offspring law has  Pareto tail. Let
$\Lambda(a):=\sup\limits_{t\in\mathbb{R}}\{at-\log\mathbb{E}[e^{tX}]\}$ be the rate function in Cram\'er's theorem; see \cite[Sec 2.2]{Dembo}.
\begin{theorem}\label{main1}Take $p\in(\nu(A),1)$ such that $I_A(\cdot)$ is continuous at $p$ and $I_A(p)<\infty$. Suppose $\mathbb{P}(|Z_1|>x)=\Theta(1)x^{-\beta}$ as $x\rightarrow+\infty$ for some constant $\beta>1$. Assume $\mathbb{E}\left[e^{\theta X}\right]<\infty$ for some $\theta>0$.\\
(i) If $0<\beta-1<-\frac{\log {p_1}}{\log m}$, then
\begin{align}
-\inf_{h>0}\Phi_1(h)
& \leq\liminf_{n\rightarrow\infty}\frac{1}{\sqrt n}\log\mathbb{P}(\bar{Z}_n(\sqrt nA)\geq p)\cr
&\leq\limsup_{n\rightarrow\infty}\frac{1}{\sqrt n}\log\mathbb{P}(\bar{Z}_n(\sqrt nA)\geq p)\leq-\sup_{h>0}\Phi'_1(h),\nonumber
\end{align}
where $\Phi_1(h):=h(\beta-1)\log m+ \Lambda\left(\frac{I_A(p)}{h}\right)h$, $\Phi'_1(h):=(h(\beta-1)\log m)\wedge \left(\Lambda\left(\frac{I_A(p)}{h}\right)h\right)$, and we make the convention that $-\log{p_1}=+\infty$ if $p_1=0$.\\
(ii) If $\beta-1\geq -\frac{\log {p_1}}{\log m}$, then
\begin{align}
-\inf_{h>0}\Phi_2(h)
&\leq\liminf_{n\rightarrow\infty}\frac{1}{\sqrt n}\log\mathbb{P}(\bar{Z}_n(\sqrt nA)\geq p)\cr
&\leq\limsup_{n\rightarrow\infty}\frac{1}{\sqrt n}\log\mathbb{P}(\bar{Z}_n(\sqrt nA)\geq p)
\leq-\sup_{h>0}\Phi'_2(h),\nonumber
\end{align}
where $\Phi_2(h):=-h\log{p_1}+\Lambda\left(\frac{I_A(p)}{h}\right)h$, $\Phi'_2(h):=(-h\log{p_1})\wedge\left(\Lambda\left(\frac{I_A(p)}{h}\right)h\right)$.
\end{theorem}
\begin{remark}If $p_1>0$, $-\log{p_1}/\log m$  is the so-called Schr\"{o}der constant, which determines the asymptotic behaviour of the harmonic moments of $|Z_n|$ (see Lemma \ref{Neyharmonic} below). Furthermore, we will see that the asymptotic behaviour of $\mathbb{P}(\bar{Z}_n(\sqrt nA)\geq p)$ mainly depends on the harmonic moments.
\end{remark}
\begin{remark}
If $\beta\in(0,1)$, Athreya and Hong \cite{Athreya coalscence} showed that $\bar{Z}_n(\sqrt n(-\infty,y])$ converges in distribution to a Bernoulli random variable.
\end{remark}

The next theorem considers the case that the offspring law has Weibull tail. As we can see in the following, the decay scales are the same as the case that offspring law has exponential moment. However, the B\"{o}ttcher constant would appear in the rate function.
\begin{theorem}\label{main5}Take $p\in(\nu(A),1)$ such that $I_A(\cdot)$ is continuous at $p$ and $I_A(p)<\infty$. Suppose $\mathbb{P}(|Z_1|>x)\sim l_1e^{-l x^{\beta}}$ as $x\rightarrow\infty$ for some constants $\beta\in(0,1)$ and $l_1,~l\in(0,+\infty)$. Assume $p_1=0$.  \\
 (i) If $\text{ess}\sup X=L\in(0,+\infty)$, then
\begin{equation}\label{weibullbounded}
\lim_{n\rightarrow\infty}\frac{1}{\sqrt{n}}\log\left[-\log \mathbb{P}(\bar{Z}_n(\sqrt nA)\geq p)\right]=\frac{I_A(p)}{L} \frac{\rho\beta}{\beta+\rho-\beta\rho}\log m,
\end{equation}
where $\rho$ is the so-called B\"{o}ttcher constant such that $b=m^{\rho}$.\\
(ii) If $\mathbb{P}(X>x)=\Theta(1)e^{-\lambda x^{\alpha}}$ as $x\rightarrow\infty$ for some constants $\alpha\in(0,\infty)$ and $\lambda>0$, then
\begin{equation}\label{weibullweibull}
\lim\limits_{n\rightarrow\infty}\frac{(\log n)^{(\alpha-1)\vee 0}}{n^{\alpha/2}}
\log\mathbb{P}(\bar{Z}_n(\sqrt nA)\geq p)= -\lambda {I_A(p)}^{\alpha}\left(\frac{2\beta\rho\log m}{\alpha(\beta+\rho-\beta\rho)}\right)^{(\alpha-1)\vee 0}.
\end{equation}
\end{theorem}
\begin{remark}For the Schr\"{o}der case (i.e., $p_1>0$), one can check that: combining the proof of \cite[Theorem 1.1]{ChenHe} and Lemma \ref{Weibulllemma} below, on can generalize \cite[Theorem 1.1]{ChenHe} to the case that $|Z_1|$ has  Weibull tail.
\end{remark}
As we can see in the above, when the step size has exponential moment, the decay rate is very sensitive to the tail of offspring law. However, in the next we will see that when the step size has Pareto tail, the offspring law seems to have little effect on the decay rate.
\begin{theorem}\label{main3}Take $p\in(\nu(A),1)$ such that $I_A(\cdot)$ is continuous at $p$ and $I_A(p)<\infty$. Suppose   $\mathbb{E}[|Z_1|^{\beta}]<\infty$, for some $\beta>1$ and $b<B$. Assume $\mathbb{P}(X>x)\sim \kappa x^{-\alpha}$ as $x\rightarrow\infty$ for some constants $\kappa>0$ and $\alpha>2$. Then
$$\lim\limits_{n\rightarrow\infty}\frac{1}{\log n}\log\mathbb{P}(\bar{Z}_n(\sqrt nA)\geq p)=-\frac{\alpha}{2}.$$\end{theorem}
\begin{remark} If $b=B$, the above result is still true, provided $A$ is unbounded and $p\in\left(\nu(A),\nu(A)+b^{-1}\left(1-\nu(A)\right)\right)$.
\end{remark}
The above theorems all assume that $I_A(p)<\infty$, and we have seen, in this case, the law of step size plays an important role on the decay rate. However, in the following we will see that if $I_A(p)=\infty$, the decay rate mainly depends on the offspring law.
\begin{theorem}\label{main4} Suppose $I_A(p)=\infty$ and $J_A(\cdot)$ is continuous at $p$ for $p\in(\nu(A),1)$. \\
(i) If $\mathbb{P}(|Z_1|>x)=\Theta(1)x^{-\beta}$ as $x\rightarrow+\infty$ for some constant $\beta>1$, then
$$
\lim\limits_{n\rightarrow\infty}\frac{1}{n}\log\mathbb{P}(\bar{Z}_n(\sqrt nA)\geq p)
=\begin{cases}
-J_A(p)(\beta-1)\log m, &0<\beta-1<-\frac{\log{p_1}}{\log m}; \cr
J_A(p)\log{p_1}, &\beta-1\geq -\frac{\log{p_1}}{\log m}.
\end{cases}
$$
(ii) Assume $\mathbb{P}(|Z_1|>x)\sim l_1e^{-l x^{\beta}}$ as $x\rightarrow\infty$ for some constants $\beta\in(0,1)$ and $l_1,~l\in(0,+\infty)$.\par
(iia) If $p_1>0$, then
$$
\lim\limits_{n\rightarrow\infty}\frac{1}{n}\log\mathbb{P}(\bar{Z}_n(\sqrt nA)\geq p)=J_A(p)\log{p_1}.
$$\par
(iib) If $p_1=0$, then
$$\lim_{n\rightarrow\infty}\frac{1}{n}\log\left[-\log \mathbb{P}(\bar{Z}_n(\sqrt nA)\geq p)\right]=J_A(p) \frac{\rho\beta}{\beta+\rho-\beta\rho}\log m.
$$
\end{theorem}
\begin{remark}In fact, Theorem \ref{main3} and Theorem \ref{main4} can also be generalized to the case that the step size $X$ is in the domain attraction of an $\alpha$-stable law with $\alpha\in(0,2]$. The results and proofs are similar: the mainly changes are to replace $\nu$ with the $\alpha$-stable law and replace $\sqrt n$ with $\inf\{x:\mathbb{P}(|X|>x)<n^{-1}\}$ (which is a regular variation sequence with index $1/\alpha$).
\end{remark}
\textbf{Idea of proofs}. Although, the proofs of above four theorems differ from cases to cases, the basic strategy behind them is the same. Let us discuss briefly about it. If $I_A(p)<\infty$, we let $x$ be the number realizing the infimum in the definition of $I_A(p)$, otherwise we let $x$ realizes the supremun in the definition of $J_A(p)$. The lower bound aims at achieving the event $\{\bar{Z}_n(\sqrt nA)\geq p)\}$ in the most effortless way. For Theorem \ref{main1}, \ref{main5} and \ref{main4}, our strategy is to let one particle reach around $x\sqrt n$ at some intermediate generation $t_n$, and then force its children to dominate the population size at time $t_n+1$ (since in these theorems, offspring law has heavy tail, this can be made with relative high probability). Finally, optimizing for $t_n$ yields the desired lower bound. For Theorem \ref{main3}, the step size has heavy tail, so particles can reach high position in a short time. Therefore, in order to achieve $\{\bar{Z}_n(\sqrt nA)\geq p)\}$, we can let one particle reach around $x\sqrt n$ in the first generation and other particles stay around the origin.\par
For the upper bound, all theorems use a similar idea on the whole, which is borrowed from Louidor and Perkins \cite{Louidor}. Our main tasks are to generalize their Lemma 2.4 in the heavy tail case and to study the asymptotic behaviours of the harmonic moments and stretched exponential moments of $|Z_n|$.\par
The rest of this paper is organised as follows. In Section \ref{prelimars}, we present some preliminary results, which will be frequently used in our proofs. We consider the offspring law has Pareto tail in Section \ref{offspringpareto}, and has Weibull tail in Section \ref{offspringweibull}. Section \ref{Paretostepsize} is devoted to study the case that the step size has Pareto tail. The last section considers the case that  $I_A(p)=+\infty$. In this paper, we always use $C$, $C'$, $C_0$, $C_1$... and $c_1$, $c_2$,...to denote positive constants. And, as usual, denote by $C(\epsilon,M)$ (or $C_{\epsilon,M}$) a positive constant depending only on $\epsilon$ and $M$. And by convention, $f(x)=\Theta(1)g(x)$ as $x\rightarrow+\infty$ means there exist constants  $C\geq C'>0$ such that $C'\leq|f(x)/g(x)|\leq C$ for all $x>1$. $f(x)\sim g(x)$ means $\lim\limits_{x\rightarrow\infty} f(x)/g(x)=1$.
\section{Preliminaries}\label{prelimars}
In the following, we mainly concentrate on the offspring law $|Z_1|$ (or step size $X$) has two typical heavy tails: Weibull tail and Pareto tail. For convenience, we write $|Z_1|\sim$Pareto$(\beta)$, if $\mathbb{P}(|Z_1|>x)=\Theta(1)x^{-\beta}$ as $x\rightarrow+\infty$ for some constant $\beta>1$. And write $|Z_1|\sim$Weibull$(\beta)$, if $\mathbb{P}(|Z_1|>x)\sim l_1e^{-l x^{\beta}}$ as $x\rightarrow\infty$ for some constants $\beta\in(0,1)$ and $l_1,~l\in(0,+\infty)$. We denote by $\nu_n$ the distribution of $S_n:=\sum^n\limits_{i=1}X_i$, where $\{X_i\}_{i\geq 1}$ are i.i.d copies of the step size $X$. Write $W_n:=|Z_n|/m^n$. From \cite[Theorem 1, Theorem 3]{Denisov GW heavy tail} and \cite[Theorem 1]{Liumoment}, we have the following uniform bounds for $W_n$.
\begin{lemma}\label{Denis}If $|Z_1|\sim$Pareto$(\beta)$ for some $\beta>1$, then there exist constants $0<c_1<c_2<\infty$ such that
    $$c_1x^{-\beta}\leq\mathbb{P}(W_n>x)\leq c_2x^{-\beta},~\text{for all}~x>1~\text{and}~n\geq 1.$$
Hence, for $\alpha\in[1,\beta)$,
$$\sup_{n}\mathbb{E}[W^{\alpha}_n]<\infty.$$
If $|Z_1|\sim$Weibull$(\beta)$ for some $\beta\in(0,1)$, then for every $\epsilon\in(0,1)$, there exist constants $c_{\epsilon}>c'_{\epsilon}>0$ depending only on $\epsilon$ such that
   $$ c'_{\epsilon}e^{-l((m+\epsilon)x)^\beta}\leq\mathbb{P}(W_n>x)\leq c_{\epsilon}e^{-l\left((m-\epsilon)x\right)^\beta},~\text{for all}~x>0~\text{and}~n\geq 1.$$
Hence, for $\alpha\in(0,\beta)$ and $\theta>0$,
$$\sup_{n}\mathbb{E}\left[e^{\theta W_n^{\alpha}}\right]<\infty.$$
If $\mathbb{E}[|Z_1|^{\beta}]<\infty$, then
\begin{equation} \label{moments}
\mathbb{E}\left[\sup_{n\geq 1}W^{\beta}_n\right]<\infty.
\end{equation}
\end{lemma}
The following lemma gives the asymptotic behaviour of harmonic moments; see \cite[Theorem 1]{Ney harmonic}.
 \begin{lemma}\label{Neyharmonic}If $\mathbb{E}[|Z_1|]<\infty$ and $r>0$, then
 $$\lim\limits_{n\rightarrow\infty}\mathbb{E}\left[{|Z_n|^{-r}}\right]{A_n(r)}=C_0,$$
where $C_0\in(0,+\infty)$ is a constant depending only on $\{p_k\}_{k\geq0}$ and $r$, and
$$
A_n(r)=\left\{
\begin{array}{rcl}
p^{-n}_1,&& {r>-\frac{\log{p_1}}{\log m}};\\
np^{-n}_1,&& {r=-\frac{\log{p_1}}{\log m}};\\
m^{nr},&& {r<-\frac{\log{p_1}}{\log m}}.
\end{array}\right.
$$
 \end{lemma}
The next lemma considers the asymptotic behaviour of stretched exponential moments, and we will see that it is related to the LDP of $\bar{Z}_n(\sqrt nA)$ when $|Z_1|$ has Weibull tail. Recall that $\rho$ is the B\"{o}ttcher constant such that $b=m^{\rho}$.
 \begin{lemma}\label{Weibullmoment}Assume $\mathbb{E}[|Z_1|\log |Z_1|]<\infty$. For any $\beta\in(0,1),~l\in(0,\infty)$,\\
 if $p_1>0$, then
$$\lim_{n\rightarrow\infty}\frac{1}{n}{\log \mathbb{E}\left[e^{-l|Z_n|^{\beta}}\right]}=p_1;$$
 if $p_1=0$, then
$$\lim_{n\rightarrow\infty}\frac{1}{n}{\log\left[-\log \mathbb{E}\left[e^{-l|Z_n|^{\beta}}\right]\right]}= \frac{\beta\rho}{\beta+\rho-\beta\rho}\log m.$$
 \end{lemma}
 \begin{proof}
 We first consider the case $p_1=0$. Let $d$ be the greatest common divisor of $\{j-k:j\neq k,~p_j p_k>0\}$. According the proof of \cite[Theorem 6]{2007Flshman}, for any fixed $\varepsilon\in(0,1)$, there exist constants $c_3\geq c_4>0$ such that for any $b^n\leq k\leq \lfloor m^{(1-\varepsilon)n}\rfloor$, $k=b^n(\text{mod}~d)$ and $n\geq 1$,
 \begin{equation}\label{asde}
 \exp\left[ -c_3\left(\frac{k}{m^n}\right)^{-\frac{\rho}{1-\rho}}\right]\leq m^n\mathbb{P}(|Z_n|=k)\leq\exp\left[ -c_4 \left(\frac{k}{m^n}\right)^{-\frac{\rho}{1-\rho}}\right].
 \end{equation}
Hence, for the upper bound, we have
\begin{align}\label{stretch}
\mathbb{E}\left[e^{-l|Z_n|^{\beta}}\right]
&\leq \mathbb{E}\left[e^{-l|Z_n|^{\beta}}1_{\{|Z_n|\leq m^{(1-\varepsilon)n}\}}\right]+e^{-lm^{\beta(1-\varepsilon)n}}\cr
&\leq m^{-n}\sum_{b^n\leq k\leq \lfloor m^{(1-\varepsilon)n}\rfloor \atop k=b^n(\text{mod}~d)}\exp(-lk^{\beta}) \exp\left[ -c_4 \left(\frac{k}{m^n}\right)^{-\frac{\rho}{1-\rho}}\right]+e^{-lm^{\beta(1-\varepsilon)n}}.
\end{align}
Note that there exists a positive constant $T$ depending only on $l,~c_4,~\rho,~\beta$ such that for $n$ large enough,
$$\min\limits_{b^n\leq x\leq\lfloor m^{(1-\varepsilon)n}\rfloor}\left\{lx^{\beta}+c_4 \left(\frac{x}{m^n}\right)^{-\frac{\rho}{1-\rho}}\right\}\geq Tm^\frac{\beta\rho n}{\beta+\rho-\beta\rho}.$$
As a consequence, if $1-\varepsilon>\frac{\rho }{\beta+\rho-\beta\rho}$ , then
 \begin{align}
 \mathbb{E}\left[e^{-l|Z_n|^{\beta}}\right]&
 \leq m^{\varepsilon n}\exp\left(-Tm^\frac{\beta\rho n}{\beta+\rho-\beta\rho}\right)+e^{-lm^{\beta(1-\varepsilon)n}}\cr
 &\leq2m^{\varepsilon n}\exp\left(-(T\wedge l)m^\frac{\beta\rho n}{\beta+\rho-\beta\rho}\right),\nonumber
 \end{align}
 which yields,
 $$\liminf_{n\rightarrow\infty}\frac{1}{n}\log\left[-\log \mathbb{E}\left[e^{-l|Z_n|^{\beta}}\right]\right]\geq \frac{\beta\rho}{\beta+\rho-\beta\rho}\log m.$$
 For the lower bound, similarly, using the left hand side of (\ref{asde}), we have for $n$ large enough,
 \begin{align}
 \mathbb{E}\left[e^{-l|Z_n|^{\beta}}\right]&\geq m^{-n}\sum_{b^n\leq k\leq \lfloor m^{(1-\varepsilon)n}\rfloor \atop k=b^n(\text{mod}~d)}\exp(-lk^{\beta}) \exp\left[ -c_3 \left(\frac{k}{m^n}\right)^{-\frac{\rho}{1-\rho}}\right]\cr
 &\geq m^{-n}\exp\left(-T'm^\frac{\beta\rho n}{\beta+\rho-\beta\rho}\right),\nonumber
 \end{align}
 where $T'>0$ depending on $l,c_3,\rho,\beta,d$. Taking limits yields that
 $$\limsup_{n\rightarrow\infty}\frac{1}{n}\log\left[-\log \mathbb{E}\left[e^{-l|Z_n|^{\beta}}\right]\right]\leq \frac{\beta\rho}{\beta+\rho-\beta\rho}\log m.$$
 For $p_1>0$, by \cite[Lemma 13]{2008Flshman}, there exists a universal constant $c>0$ such that for any $k,n\geq1$,
 $$\mathbb{P}(|Z_n|=k)\leq cp^n_1k^{-(\log p_1/\log m)-1}.$$
 Thus,
 $$p^n_1e^{-l}\leq\mathbb{E}\left[e^{-l|Z_n|^{\beta}}\right]\leq cp^n_1\sum_{k\geq1} k^{-(\log p_1/\log m)-1}e^{-lk^{\beta}},$$
 which implies the desired result.
 \end{proof}
The following two lemmas are analogous results of \cite[Lemma 2.4]{Louidor}. In the sequel, denote by $\mathcal{M}$ the collection of all local finite point measures on $\mathbb{R}$. Recall that for $\xi\in\mathcal{M}$,
 $x\in\xi$ means $x$ is an atom of $\xi$, and $S_x$ is the position of $x$ (i.e., $\xi=\sum_{x\in\xi}\delta_{S_x}$). $|\xi|$ stands for the total number of its atoms. Let $Z^{\xi}_n$ be the branching random walk started from $Z^{\xi}_0=\xi$. And for simplicity, we write $Z^{x}_n:=Z^{\delta_{S_x}}_n$. Denote by $\bar{Z}^{\xi}_n(\cdot)$ the corresponding empirical distribution of $Z^{\xi}_n$. Let $W^x_n:=|Z^{x}_n|/m^n$.
\begin{lemma} \label{louidorlem}
Assume $\mathbb{E}[|Z_1|^{\beta}]<\infty$ or $|Z_1|\sim$Pareto$(\beta)$ for some $\beta>1$. Then for every $\epsilon\in(0,\frac{1}{3})$, there exists a constant $C_{1}>0$ depending on $\epsilon$, $\beta$  such that for any $\xi\in\mathcal{M}$, $n\geq 1$ and $A\subset\mathbb{R}$,
\begin{equation}\label{louidorshi}
\mathbb{P}\Big(\bar{Z}^{\xi}_n(A)\geq \frac{1}{|\xi|}\sum_{x\in\xi}\nu_n\left(A-S_x\right)+\epsilon\Big)\leq C_{1}|\xi|^{1-\beta}.
\end{equation}
The same holds if $>$, $+\epsilon$ are replaced by $<$, $-\epsilon$.
\end{lemma}
\begin{proof}We first consider the case $\mathbb{E}[|Z_1|^{\beta}]<\infty$ for some $\beta>1$. By the branching property, for any $n\geq 1$ and $\epsilon\in(0,\frac{1}{3})$, we have
\begin{align}\label{louidor}
&\mathbb{P}\left(\bar{Z}^{\xi}_n(A)\geq \frac{1}{|\xi|}\sum_{x\in\xi}\nu_n\left(A-S_x\right)+\epsilon\right)\cr
&=\mathbb{P}\left(\frac{\frac{1}{|\xi|}\sum_{x\in\xi}{Z}^{x}_n(A)/m^n}{\frac{1}{|\xi|}\sum_{x\in\xi}{Z}^{x}_n/m^n}\geq\frac{1}{|\xi|}\sum_{x\in\xi}\nu_n(A-S_x)+\epsilon\right)\cr
&\leq\mathbb{P}\left(\sum_{x\in\xi}{W}^{x}_n<|\xi|(1-\epsilon/2)\right)+\mathbb{P}\left(\sum_{x\in\xi}\left({W}^{x}_n(A)-\nu_n\left(A-S_x\right)\right)>|\xi|\epsilon/3\right)\cr
&\leq\mathbb{P}\left(\sum_{x\in\xi}(1-W^x_n)>\frac{\epsilon}{3}{|\xi|}\right)+\mathbb{P}\left(\sum_{x\in\xi}\left({W}^{x}_n(A)-\nu_n\left(A-S_x\right)\right)>|\xi|\epsilon/3\right)\cr
&=:I_1+I_2,
\end{align}
where $\epsilon\in(0,\frac{1}{3})$ is used in the first inequality. We first treat $I_2$. By \cite[Corollary 1.6]{Nagaev}, we know: if $Y_x,~x\in\xi$ are independent random variables with zero mean, and $A^+_t=\sum_{x\in\xi}\mathbb{E}[Y^t_x 1_{\{Y_x\geq 0\}}]<\infty$ for some $1\leq t\leq 2$, then for $y^t\geq 4A^+_t$ and $z>y$,
\begin{equation}\label{Nagaevinequailty}
\mathbb{P}\left(\sum\limits_{x\in\xi} Y_x\geq z\right)\leq\sum\limits_{x\in\xi}\mathbb{P}(Y_x>y)+\left(e^2A^+_t/zy^{t-1}\right)^{z/2y}.
\end{equation}
So, if we choose $t=1+\frac{\beta-1}{2\beta}$, $y=\frac{1}{4\beta}z$ and $z=|\xi|\epsilon/3$, then $t\in(1,2\wedge\beta]$. Let $Y_x={W}^{x}_n(A)-\nu_n(A-x)$. Note that by (\ref{moments}), there exists a constant $C_{\beta}>0$ such that
\begin{equation}\label{paretomoment}
\sup_{n}\mathbb{E}[W^{t}_n]<C_{\beta}<\infty.
\end{equation}
Therefore, $A^+_t<\infty$. Furthermore, by (\ref{paretomoment}), there exists a constant $C_{\epsilon,\beta}>0$ such that for all $|\xi|>C_{\epsilon,\beta}$ and $n\geq1$,
$$y^t=\left(\frac{\epsilon|\xi|}{12\beta}\right)^t\geq 4|\xi|C_{\beta}\geq 4|\xi|\mathbb{E}[W^t_n]\geq 4\sum\limits_{x\in\xi}\mathbb{E}[Y^t_x 1_{\{Y_x\geq 0\}}].$$
So, applying (\ref{Nagaevinequailty}), we obtain that for all $|\xi|>C_{\epsilon,\beta}$, $n\geq 1$ and $A\subset\mathbb{R}$,
\begin{align}\label{dfdf}
I_2&\leq|\xi|\mathbb{P}\Big({W}^{x}_n(A)-\nu_n(A-S_x)>\frac{\epsilon|\xi|}{12\beta}\Big)+\left(\frac{e^2C_{\beta}3^{t}}{(4\beta)^{1-t}\epsilon^t}\right)^{2\beta}|\xi|^{(1-t)2\beta}\cr
 &\leq|\xi|\mathbb{P}\Big({W}_n>\frac{\epsilon|\xi|}{12\beta}\Big)+\left(\frac{e^2C_{\beta}3^{t}}{(4\beta)^{1-t}\epsilon^t}\right)^{2\beta}|\xi|^{1-\beta}.
\end{align}
By (\ref{moments}) and Markov inequality, there exists a constant $C'_{\epsilon,\beta}>0$ such that for all $n\geq1$,
\begin{equation}\label{dfdq}
 \mathbb{P}\Big({W}_n>\frac{\epsilon|\xi|}{12\beta}\Big)\leq \left(\frac{12\beta}{\epsilon|\xi|}\right)^{\beta}\mathbb{E}\left[\sup_nW^{\beta}_n\right]\leq C'_{\epsilon,\beta}|\xi|^{-\beta}.
 \end{equation}
Plugging (\ref{dfdq}) into (\ref{dfdf}), yields that for all $|\xi|>C_{\epsilon,\beta}$, $n\geq 1$ and $A\subset\mathbb{R}$,
$$I_2\leq C'_{\epsilon,\beta}|\xi|^{1-\beta}+\left(\frac{e^2C_{\beta}3^{t}}{(4\beta)^{1-t}\epsilon^t}\right)^{2\beta}|\xi|^{1-\beta}.$$
For $I_1$, using (\ref{Nagaevinequailty}) again and let $Y_x=1-W^x_n$, then for all $|\xi|>\frac{12\beta}{\epsilon}$ and $n\geq1$,
$$I_1\leq|\xi|\mathbb{P}\Big(1-W^x_n>\frac{\epsilon|\xi|}{12\beta}\Big)+\left(\frac{e^23^{t}}{(4\beta)^{1-t}\epsilon^t}\right)^{2\beta}|\xi|^{(1-t)2\beta}= \left(\frac{e^23^{t}}{(4\beta)^{1-t}\epsilon^t}\right)^{2\beta}|\xi|^{1-\beta}.$$
Plugging above two inequalities into (\ref{louidor}), yields that there exists a constant $T(\epsilon,\beta)>0$ such that for $|\xi|>C_{\epsilon,\beta}\vee\frac{12\beta}{\epsilon}$,
$$
\mathbb{P}\Big(\bar{Z}^{\xi}_n(A)\geq \frac{1}{|\xi|}\sum_{x\in\xi}\nu_n\left(A-S_x\right)+\epsilon\Big)\leq T(\epsilon,\beta)|\xi|^{1-\beta}.
$$
 Hence to obtain (\ref{louidorshi}), we can take $C_1:=\left(C_{\epsilon,\beta}\vee\frac{12\beta}{\epsilon}\right)^{\beta-1}\vee T(\epsilon,\beta)$.
 For the case $\mathbb{P}(|Z_1|>x)\sim \Theta(1)x^{-\beta}$, we only need to replace (\ref{dfdq}) by the following, which is a consequence of Lemma \ref{Denis}.
$$
\mathbb{P}\left({W}_n>\frac{\epsilon|\xi|}{12\beta}\right)\leq c_2\left(\frac{\epsilon}{12\beta}\right)^{-\beta}|\xi|^{-\beta}.
$$
  Replacing $A$ with $A^c$, we obtain (\ref{louidorshi}) with $>$, $+\epsilon$ are replaced by $<$, $-\epsilon$.
\end{proof}
\begin{remark}\label{pointlouidor} One can check that: if $\xi$ is a point process with determinate number of atoms, the result is similar, only that $\mathbb{E}[\nu_n(A-S_x)]$ replaces $\nu_n(A-S_x)$.
\end{remark}

\begin{lemma}\label{Weibulllemma}If $|Z_1|\sim$Weibull$(\beta)$ for some $\beta\in(0,1)$, then for every $\epsilon>0$ small enough
there exist positive constants $C_2$, $C_3$ depending on $\epsilon$, $\beta$ such that for any $\xi\in\mathcal{M}$, $n\geq 1$ and any $A\subset\mathbb{R}$,
\begin{equation}
\mathbb{P}\Big(\bar{Z}^{\xi}_n(A)\geq \frac{1}{|\xi|}\sum_{x\in\xi}\nu_n(A-S_x)+\epsilon\Big)\leq C_2e^{-C_3|\xi|^{\beta}}.
\end{equation}
The same holds if $>$, $+\epsilon$ are replaced by $<$, $-\epsilon$.
\end{lemma}
\begin{proof}From (\ref{louidor}), for every $\epsilon\in\left(0,\frac{1}{3}\right)$, we have
\begin{align}\label{loidorweb}
&\mathbb{P}\Big(\bar{Z}^{\xi}_n(A)\geq \frac{1}{|\xi|}\sum_{x\in\xi}\nu_n(A-S_x)+\epsilon\Big)\cr
&\leq\mathbb{P}\left(\sum_{x\in\xi}(1-W^x_n)>\frac{\epsilon}{3}{|\xi|}\right)+\mathbb{P}\left(\sum_{x\in\xi}{W}^{x}_n(A)>\sum_{x\in\xi}\nu_n(A-S_x)+|\xi|\epsilon/3\right)\cr
&=:I_1+I_2.
\end{align}
For $I_1$, note that $\sup\limits_{x,n}\mathbb{E}\left[e^{1-W^x_n}\right]<\infty$. Furthermore, by Lemma \ref{Denis},
$$\mathbb{E}\left[((1-W^x_n)^-)^2\right]\leq 3+\sup_n\mathbb{E}\left[W^2_n\right]<\infty.$$
Hence, by \cite[Lemma 2.3]{Louidor}, there exists a constant $c_5>0$ such that for all $n$, $\xi$ and $\epsilon$ small enough,
\begin{equation}\label{www}
 I_1=\mathbb{P}\left(\sum_{x\in\xi}(1-W^x_n)>\frac{\epsilon}{3}{|\xi|}\right)\leq e^{-c_5\epsilon^2|\xi|}.
 \end{equation}
 For $I_2$, by Lemma \ref{Denis}, there exists a constant $c_6>0$ such that for all $y>0$ and $n\geq1$,
 $$\mathbb{P}({W}^{x}_n(A)>y)\leq\mathbb{P}(W_n>y)\leq c_6e^{-l(m-1)^{\beta}y^\beta}.$$
 Let $X_i$, $i\geq 1$ be i.i.d copies of ${W}^{x}_n(A)$, and
 $$
a_i(|\xi|)=
\begin{cases}
\left(|\xi|\epsilon/3+\sum\limits_{x\in\xi}\nu_n(A-S_x)\right)^{-1},&~1\leq i\leq |\xi|;\\
~~~~~~~~~~~~~~~~0,&~i> |\xi|.
\end{cases}.
 $$
 Then by slight modifications  of the proof of upper bound \cite[Theorem 2.1]{FrankWebull}, there exist positive constants $T$, $T'$ depending on $\epsilon$, $\beta$ such that for any $|\xi|>T$, $n\geq 1$ and any $A\subset\mathbb{R}$,
 \begin{equation}\label{wew}
 I_2=\mathbb{P}\Big(\sum_{x\in\xi}{W}^{x}_n(A)>\sum_{x\in\xi}\nu_n(A-S_x)+|\xi|\epsilon/3\Big)\leq e^{-T'|\xi|^{\beta}}.
 \end{equation}
 Plugging (\ref{www}) and (\ref{wew}) into (\ref{loidorweb}), concludes this lemma with $C_2:=e^{T'T^{\beta}}$  and $C_3:=T'\wedge(c_5\epsilon^2)$.
\end{proof}
The following lemma concerns large deviation probabilities of sums of i.i.d Weibull tail random variables.
\begin{lemma}\label{WeibullLDP} Suppose $\{X_i\}_{i\geq 1}$ is a sequence of i.i.d. random variables and $\mathbb{P}(X>x)=\Theta(1)e^{-\lambda x^{\alpha}}$ as $x\rightarrow\infty$ with some $\alpha\in(1,\infty)$ and $\lambda>0$. Assume $t_n=o(n^{1/3})$ and $t_n\rightarrow\infty$. For any $0<a<b\leq+\infty$, we have
$$\lim\limits_{n\rightarrow\infty}\frac{{t_n}^{\alpha-1}}{n^{\alpha/2}}
\log\mathbb{P}\left(\sum^{t_n}_{i=1}X_i\in(a\sqrt n,b\sqrt n)\right)= -\lambda a^{\alpha}.$$
\end{lemma}
\begin{proof}The upper bound can be found in \cite[Lemma B.1]{ChenHe}. For the lower bound, since $Ce^{-\lambda x^{\alpha}}\leq\mathbb{P}(X>x)\leq C'e^{-\lambda x^{\alpha}}$ for some constants $C'>C>0$, we have for $n$ large enough,
\begin{align}
\mathbb{P}\left(\sum^{t_n}_{i=1}X_i\in(a\sqrt n,b\sqrt n)\right)&\geq\mathbb{P}\left(X_i\in\left(\frac{a\sqrt n}{t_n},\frac{b\sqrt n}{t_n}, \right), \forall 1\leq i\leq t_n\right)\cr
&\geq \left(Ce^{-\lambda(\frac{a\sqrt n}{t_n})^{\alpha}}-C'e^{-\lambda(\frac{b\sqrt n}{t_n})^{\alpha}}\right)^{t_n}\cr
&\geq C^{t_n}e^{-\lambda a^{\alpha}\frac{n^{\alpha/2}}{t^{\alpha-1}_n}}.\nonumber
\end{align}
Then, the desired lower bound follows, provided $t_n=o(\sqrt n)$ .
\end{proof}
\section{Proof of Theorem \ref{main1}.}\label{offspringpareto}
In this section, we assume $|Z_1|\sim$Pareto$(\beta)$ for some $\beta\in(1,+\infty)$, and $\mathbb{E}[e^{\theta X}]<\infty$ for some $\theta>0$. We also assume that $I_A(p)<\infty$ and $I_A(\cdot)$ is continuous at $p$.
\begin{proof}
\textbf{Lower bound}. Fix $\epsilon>0$. By the continuity of $I_A(\cdot)$ at $p$, there exist some $\delta>0$ and $|x|<I_A(p)+\epsilon$ such that for any small $\eta>0$,
    $$\inf_{y\in[x-\eta,x+\eta]}\nu(A-y)\geq p+\delta.$$
 Consequently, we can choose $M$ large enough such that
 $$\frac{1}{1+M^{-1}}\inf_{y\in[x-\eta,x+\eta]}\nu(A-y)\geq p+\frac{\delta}{2}.$$
 Let $t_n:=\lfloor h\sqrt n\rfloor$ with some $h>0$ and $|v|$ be the generation of particle $v$. Set
 \begin{align}
  Z^v_{k}&:=\sum_ {u\in Z_{|v|+k},~u~\text{is a descent of}~v}\delta_{S_u-S_v};\cr
  H(u)&:=\left\{w\in Z_{t_n+1}:~w\notin Z^u_{1}\right\},~\text{for}~u\in Z_{t_n};\cr
\mathcal{E}&:=\left\{(\xi,k,r)\in \mathcal{M}\times\mathbb{N^+}\times\mathbb{N^+}: \xi~\text{has exactly one atom}~z~\text{s.t.}~S_z\in[(x-\eta)\sqrt n,(x+\eta)\sqrt n];~k>2Mr\right\};\cr
E&:=\left\{Z_{t_n} \text{has exactly one particle}~u~\text{s.t.}~S_u~\text{in}~[(x-\eta)\sqrt n,(x+\eta)\sqrt n]~\text{and}~|Z^u_1|>2M\sum_{v\neq u,v\in Z_{t_n}}|Z^v_1|\right\}.\nonumber
 \end{align}
 Namely, $Z^v_{k}$ is the $k$th generation of the sub-BRW emanating form particle $v$. And $H(u)$ stands for a collection of particles at time $t_n+1$, who are not the children of $u$. The proof of the lower bound is mainly divided into two steps.\par
 \textbf{Step 1}. In this step, we will show that there exists a constant $C_M>0$ such that for $n$ large enough, $\mathbb{P}\Big(\bar{Z}_n(\sqrt nA)\geq p\Big)\geq C_M\mathbb{P}(E)$. Let $\{Z^{i}_n\}_{n\geq 0}, i\geq 1$ be i.i.d copies of $\{Z_n\}_{n\geq 0}$. Let $x_v$ be the displacement of particle $v$ and $W^{v}_{n-t_n-1}:=|Z^{v}_{n-t_n-1}|m^{-(n-t_n-1)}$. By Markov property, we have
\begin{align}\label{paretolow}
&\mathbb{P}\Big(\bar{Z}_n(\sqrt nA)\geq p\Big)\cr
&\geq\mathbb{P}\left(\exists u\in Z_{t_n},\frac{\sum\limits_{v\in Z^u_{1}}Z^{v}_{n-t_n-1}(\sqrt nA-x_v-S_u)+
\sum\limits_{w\in H(u)}
Z^{w}_{n-t_n-1}(\sqrt nA-S_w)}
{\sum\limits_{v\in Z^u_{1}}|Z^{v}_{n-t_n-1}|+
\sum\limits_{w\in H(u)}
|Z^{w}_{n-t_n-1}|}\geq p,\frac{\sum\limits_{v\in Z^u_{1}}|Z^{v}_{n-t_n-1}|}{\sum\limits_{w\in H(u)}
|Z^{w}_{n-t_n-1}|}>M \right)\cr
&\geq\displaystyle {\int_{\mathcal{E}}\mathbb{P}\left(\frac{\sum^k\limits_{i=1}Z^{i}_{n-t_n-1}(\sqrt nA-x_v-S_z)}
{\sum^k\limits_{i=1}|Z^{i}_{n-t_n-1}|+
\sum^r\limits_{j=1}
|Z^{j}_{n-t_n-1}|}\geq p,\frac{\sum^k\limits_{i=1}|Z^{i}_{n-t_n-1}|}{\sum^r\limits_{j=1}
|Z^{j}_{n-t_n-1}|}>M \right)}\mathbb{P}\left(Z_{t_n}\in d\xi,|Z^{z}_{1}|=k,|H(z)|=r\right)\cr
&\geq\int_{\mathcal{E}}\mathbb{P}\left(\frac{\sum^k\limits_{i=1}Z^{i}_{n-t_n-1}(\sqrt nA-x_v-S_z)}{(1+M^{-1})\sum^k\limits_{i=1}|Z^{i}_{n-t_n-1}|}\geq p, \frac{\sum^k\limits_{i=1}{W^{i}_{n-t_n-1}
}}{\sum^r\limits_{j=1}W^{j}_{n-t_n-1}}>M \right)\mathbb{P}\left(Z_{t_n}\in d\xi,|Z^{z}_{1}|=k, |H(z)|=r\right)\cr
&=:\int_{\mathcal{E}} P_n(\xi,k,r)\mathbb{P}\left(Z_{t_n}\in d\xi, |Z^{z}_{1}|=k, |H(z)|=r\right),
\end{align}
where $\mathbb{P}(Z_{t_n}\in d\xi)$ stands for the distribution of the point process $Z_{t_n}$ (for serious definition of  point process's distribution see \cite[Sec 2.1]{Anton}). To finish Step 1, it suffices to show that
\begin{equation}\label{C}
\lim\limits_{n\rightarrow\infty}\inf\limits_{(\xi,k,r)\in \mathcal{E}}P_n(\xi,k,r)>0.
\end{equation}
Since $A\in\mathcal{A}$ (see Assumption \ref{assume}), we can write $A=\sum\limits^l_{i=1}(a_i,b_i]$ for some natural number $l$, where $-\infty\leq a_i<b_i\leq+\infty$. Let $$A(x,\eta):=\sum\limits^l_{i=1}\left(a_i-x+\eta,b_i-x-\eta\right).$$
 Since $S_z\in[(x-\eta)\sqrt n,(x+\eta)\sqrt n]$, we can choose $\eta$ small enough such that
\begin{equation}\label{Ae}
\frac{1}{1+M^{-1}}\nu(A(x,\eta))>p+\frac{\delta}{4}~\text{and}~\sqrt nA(x,\eta)\subset \sqrt nA-S_z.
\end{equation}
Thus, by central limit theorem, there exists a constant $C(M,h,x,\eta,\delta)>0$ such that for $n>C(M,h,x,\eta,\delta)$,
\begin{equation}\label{sdewd}
\frac{1}{1+M^{-1}}\nu_{n-t_n}(\sqrt nA(x,\eta))>p+\frac{\delta}{8}.
\end{equation}
 Recall that $Z^{i}_{n-t_n-1}$, $W^{i}_{n-t_n-1}$ and $X_i$, $1\leq i\leq k$ are respectively i.i.d copies of $Z_{n-t_n-1}$, $W_{n-t_n-1}$ and $X$. Since $\sqrt nA(x,\eta)\subset \sqrt nA-S_z$, by (\ref{paretolow}), we have
$$P_n(\xi,k,m)\geq \mathbb{P}\left(\frac{\sum\limits^k_{i=1}Z^{i}_{n-t_n-1}(\sqrt nA(x,\eta)-X_i)}{(1+M^{-1})\sum\limits^k_{i=1}|Z^{i}_{n-t_n-1}|}\geq p,~\frac{\sum\limits^k_{i=1}{W^{i}_{n-t_n-1}
}}{\sum\limits^{\lfloor\frac{k}{2M}\rfloor}_{j=1}W^{j}_{n-t_n-1}}>M\right)=:\mathbb{E}\left[\ind_{\{A_{n,k}\geq p\}}\ind_{\{B_{n,k}\geq M\}}\right].$$
Fix $k> 2M$ (since $(\xi,k,r)\in\mathcal{E}$) and $\epsilon'\in(0,\eta)$. There exists some random variable $N(\epsilon', k)>0$ such that for $n>N(\epsilon', k)$ and $1\leq i\leq k$,
$$\sqrt nA(x,\eta+\epsilon')\subset\sqrt nA(x,\eta)-X_i\subset \sqrt nA(x,\eta-\epsilon').$$
Thus
      $$\frac{\sum\limits^k_{i=1}m^{-(n-t_n-1)}Z^{i}_{n-t_n-1}(\sqrt nA(x,\eta+\epsilon'))}{(1+M^{-1})\sum\limits^k_{i=1}m^{-(n-t_n-1)}|Z^{i}_{n-t_n-1}|}\leq A_{n,k}\leq \frac{\sum\limits^k_{i=1}m^{-(n-t_n-1)}Z_{n-t_n-1}^{i}(\sqrt nA(x,\eta-\epsilon'))}{(1+M^{-1})\sum\limits^k_{i=1}m^{-(n-t_n-1)}|Z^{i}_{n-t_n-1}|}.$$
Since $\bar{Z}_n(\sqrt nA)\rightarrow\nu(A)$ and $W_n\rightarrow W$, the above implies that
$$\lim\limits_{n\rightarrow \infty}A_{n,k}= \frac{1}{1+M^{-1}}\nu(A(x,\eta)),~ \mathbb{P}-a.s.$$
On the other hand, it is easy to see that
$$~~\lim\limits_{n\rightarrow\infty}B_{n,k}=\frac{\sum\limits^k_{i=1}{W^{i}
}}{\sum\limits^{\lfloor\frac{k}{2M}\rfloor}_{j=1}W^{j}}=:B_k,~~\mathbb{P}-a.s.,$$
where $W^i$ and $W^j$ are i.i.d copies of $W:=\lim\limits_{n\rightarrow\infty}|Z_n|/m^n$. Hence, by the dominated convergence theorem and (\ref{Ae}), for any fixed $k>2M$, we have
\begin{equation}\label{as}
\lim\limits_{n\rightarrow \infty}\mathbb{E}\left[\ind_{\{A_{n,k}\geq p\}}\ind_{\{B_{n,k}\geq M\}}\right]=\mathbb{P}\left(B_k\geq M\right).
\end{equation}
 So, to achieve (\ref{C}), it suffices to show:
 \begin{equation}\label{uni}
 \lim\limits_{n\rightarrow\infty}\sup_{k> 2M}\left|\mathbb{E}\left[\ind_{\{A_{n,k}\geq p\}}\ind_{\{B_{n,k}\geq M\}}\right]-\mathbb{P}\left(B_k\geq M\right)\right|=0,
\end{equation}
and
\begin{equation}\label{dfdfsd}
\inf_{k>2M}\mathbb{P}\left(B_k>M\right)>0.
\end{equation}
By the strong law of large numbers, one can easily obtain $\inf\limits_{k>2M}\mathbb{P}\left(B_k>M\right)>0$. For (\ref{as}), observe that
\begin{align}\label{ap}
&\left|\mathbb{E}\left[\ind_{\{A_{n,k}\geq p\}}\ind_{\{B_{n,k}\geq M\}}\right]-\mathbb{P}\left(B_k\geq M\right)\right|\cr
&\leq\mathbb{P}(A_{n,k}<p)+\mathbb{P}(B_{n,k}\geq M,B_k<M)+\mathbb{P}(B_{n,k}<M,B_k\geq M).
\end{align}
For the first term on the r.h.s of (\ref{ap}), let $\xi=\sum^k_{i=1}\delta_{X_i}$ be a point process (recall that $X_i$, $i\geq 1$ are i.i.d copies of $X$). By Remark \ref{pointlouidor} and (\ref{sdewd}), there exists a constant $C_1>0$ depending on $\delta$, $\beta$ such that for $n>C(M,h,x,\eta,\delta)$,
\begin{align}\label{dferfwer}
\mathbb{P}(A_{n,k}<p)&=\mathbb{P}\left(\frac{\sum\limits^k_{i=1}Z^{i}_{n-t_n-1}(\sqrt nA(x,\eta)-X_i)}{(1+M^{-1})\sum\limits^k_{i=1}|Z^{i}_{n-t_n-1}|}<p \right)\cr
&=\mathbb{P}\left(\bar{Z}^{\xi}_{n-t_n-1}(\sqrt nA(x,\eta))<p(1+M^{-1})\right)\cr
&\leq \mathbb{P}\left(\bar{Z}^{\xi}_{n-t_n-1}(\sqrt nA(x,\eta))<\nu_{n-t_n}(\sqrt nA(x,\eta))-\frac{\delta}{8}\right)\cr
&=\mathbb{P}\left(\bar{Z}^{\xi}_{n-t_n-1}(\sqrt nA(x,\eta))<\mathbb{E}\left[\frac{1}{|\xi|}\sum_{z\in\xi}\nu_{n-t_n-1}(\sqrt nA(x,\eta)-S_z)\right]-\frac{\delta}{8}\right)\cr
&\leq C_1k^{1-\beta},
\end{align}
where the third equality holds since
 $$\nu_{n-t_n}\left(\sqrt nA(x,\eta)\right)=\mathbb{P}\left(X_1+X_2+...+X_{n-t_n}\in\sqrt nA(x,\eta)\right)=\mathbb{E}\left[\nu_{n-t_n-1}(\sqrt nA(x,\eta)-S_z)\right].$$
For the second term on the r.h.s of (\ref{ap}), by the strong law of large numbers, we have
\begin{equation}\label{dferfwer1}
\mathbb{P}(B_{n,k}\geq M,B_k<M)\leq\mathbb{P}(B_k<M)\rightarrow 0,~\text{as}~ k\rightarrow\infty.
\end{equation}
For the third term on the r.h.s of (\ref{ap}), there exists a constant $C(M,\beta)>0$ such that for any $k>2M$ and $n\geq1$,
\begin{align}\label{dferfwer2}
&\mathbb{P}(B_{n,k}<M,B_k\geq M)\cr
&\leq\mathbb{P}(B_{n,k}<M)\cr
&=\mathbb{P}\left(\frac{\sum\limits^k_{i=1}{W^{i}_{n-t_n-1}
}}{\sum\limits^{\lfloor\frac{k}{2M}\rfloor}_{j=1}W^{j}_{n-t_n-1}}<M\right)\cr
&\leq\mathbb{P}\left(\sum\limits^k_{i=1}W^{i}_{n-t_n-1}<M\sum\limits^{\lfloor\frac{k}{2M}\rfloor}_{j=1}W^{j}_{n-t_n-1},
\sum\limits^{\lfloor\frac{k}{2M}\rfloor}_{j=1}W^{j}_{n-t_n-1}<\frac{k}{2M}\frac{3}{2}\right)+
\mathbb{P}\left(\sum\limits^{\lfloor\frac{k}{2M}\rfloor}_{j=1}W^{j}_{n-t_n-1}\geq\frac{k}{2M}\frac{3}{2}\right)\cr
&\leq\mathbb{P}\left(\sum\limits^k_{i=1}W^{i}_{n-t_n-1}<\frac{3k}{4}\right)+
\mathbb{P}\left(\sum\limits^{\lfloor\frac{k}{2M}\rfloor}_{j=1}W^{j}_{n-t_n-1}\geq\frac{k}{2M}\frac{3}{2}\right)\cr
&\leq\mathbb{P}\left(\sum\limits^k_{i=1}\left(1-W^{i}_{n-t_n-1}\right)>\frac{k}{4}\right)+
\mathbb{P}\left(\sum\limits^{\lfloor\frac{k}{2M}\rfloor}_{j=1}\left(W^{j}_{n-t_n-1}-1\right)\geq\frac{k}{2M}\frac{1}{2}\right)\cr
&\leq C(M,\beta)k^{1-\beta},
\end{align}
where for the last inequality, we use exactly the same arguments as bounding $I_1$ and $I_2$ in Lemma \ref{louidorlem}. Plugging (\ref{dferfwer}), (\ref{dferfwer1}) and (\ref{dferfwer2}) into (\ref{ap}), we obtain that for every $\epsilon>0$ there exists a constant $C(\epsilon,M,\beta,\delta)$ such that for $k>C(\epsilon,M,\beta,\delta)$ and $n>C(M,h,x,\eta,\delta)$,
$$\left|\mathbb{E}\left[\ind_{\{A_{n,k}\geq p\}}\ind_{\{B_{n,k}\geq M\}}\right]-\mathbb{P}\left(B_k\geq M\right)\right|<\epsilon.$$
Combining this with (\ref{as}), there exists some large constant $C(\epsilon,M,\beta,h,x,\eta,\delta,p)$ such that for $n>C(\epsilon,M,\beta,h,x,\eta,\delta,p)$ and $k>2M$,
$$\left|\mathbb{E}\left[\ind_{\{A_{n,k}\geq p\}}\ind_{\{B_{n,k}\geq M\}}\right]-\mathbb{P}\left(B_k\geq M\right)\right|<\epsilon$$
Thus (\ref{uni}) holds. This, combined with (\ref{dfdfsd}) and (\ref{paretolow}), implies that there exist positive constants $C_M$ and $C(M,\beta,h,x,\eta,\delta,p)$ such that for $n>C(M,\beta,h,x,\eta,\delta,p)$,
$$\mathbb{P}\Big(\bar{Z}_n(\sqrt nA)\geq p\Big)\geq C_M\mathbb{P}(E),$$
which completes Step 1.\par
 \textbf{Step 2}. In this step, we will give a lower bound of $\mathbb{P}(E)$. Define $\mathcal{F}_{t_n}:=\sigma\left(Z_{i},~1\leq i\leq t_n\right)$. By Step 1, for $n>C(M,\beta,h,x,\eta,\delta,p)$,
 \begin{align}\label{bb}
\mathbb{P}\left(\bar{Z}_n(\sqrt nA)\geq p\right)
&\geq C_M\mathbb{P}\left(\exists u\in Z_{t_n},S_u\in[(x-\eta)\sqrt n,(x+\eta)\sqrt n],|Z^{u}_{1}|>2M\sum_{v\neq u,v\in Z_{t_n}}|Z^{v}_{1}|\right)\cr
&=C_M\mathbb{E}\left[\sum_{u\in Z_{t_n}}\ind_{\{S_u\in[(x-\eta)\sqrt n,(x+\eta)\sqrt n]\}}\ind_{\{|Z^{u}_{1}|>2M\sum\limits_{v\neq u, v\in Z_{t_n}}|Z^{v}_{1}|\}}\right]\cr
&=C_M\mathbb{E}\left[\sum_{u\in Z_{t_n}}\ind_{\{S_u\in[(x-\eta)\sqrt n,(x+\eta)\sqrt n]\}}\mathbb{E}\left[\ind_{\{|Z^{u}_{1}|>2M\sum\limits_{v\neq u, v\in Z_{t_n}}|Z^{v}_{1}|\}}\Big{|}\mathcal{F}_{t_n}\right]\right],
\end{align}
where the first equality follows from the fact that the random variable inside the expectation can only be 0 or 1. Let $k_i$, $i\geq 0$ be i.i.d copies of $|Z_1|$, and independent with $Z_{t_n}$. Since $|Z_1|\sim$Pareto($\beta$), there exists a constant $C_4>0$ such that $\mathbb{P}(|Z_1|>x)>C_4 x^{-\beta}$, for all $x>1$. Recall a well known fact: if $U$, $V$ are independent random variables, then for any bounded measurable function $F(x,y)$, we have
 $$\mathbb{E}[F(U,V)|\sigma(V)]=\mathbb{E}[F(U,v)]|_{v=V}.$$
Using this fact, we have
\begin{align}\label{sdfaerf}
\mathbb{E}\left[\ind_{\{|Z^{u}_{1}|>2M\sum\limits_{v\neq u, v\in Z_{t_n}}|Z^{v}_{1}|\}}\Big{|}\mathcal{F}_{t_n}\right]&=
\mathbb{E}\left[\ind_{\{k_0>2M\sum^{|Z_{t_n}|-1}_{i=1}k_i\}}\Big{|}\mathcal{F}_{t_n}\right]\cr
&=\mathbb{E}\left[\ind_{\{k_0>2M\sum^{j}_{i=1}k_i\}}\right]\Big{|}_{j=|Z_{t_n}|-1}\cr
&=\mathbb{E}\left[\mathbb{E}\left[\ind_{\{k_0>2M\sum^{j}_{i=1}k_i\}}\Big{|}\sigma(k_1,...,k_j)\right]\right]\Big{|}_{j=|Z_{t_n}|-1}\cr
&\geq C_4\mathbb{E}\left[\left(2M\sum^{j}_{i=1}k_i\right)^{-\beta}\right]\Big{|}_{j=|Z_{t_n}|-1}\cr
&= C_4\mathbb{E}\left[\left(2M\sum^{|Z_{t_n}|-1}_{i=1}k_i\right)^{-\beta}\Big{|}\sigma(|Z_{t_n}|)\right]\cr
&\geq C_4(2M)^{-\beta}\mathbb{E}\left[  |Z_{t_n+1}|^{-\beta}\Big{|}\mathcal{F}_{t_n}\right].
\end{align}
Plugging (\ref{sdfaerf}) into (\ref{bb}), yields that for $n>C(M,\beta,h,x,\eta,\delta,p)$,
\begin{align}\label{sadfefw}
\mathbb{P}\left(\bar{Z}_n(\sqrt nA)\geq p\right)&\geq C_MC_4(2M)^{-\beta}\mathbb{E}\left[\mathbb{E}\left[\sum_{u\in Z_{t_n}}\ind_{\{S_u\in[(x-\eta)\sqrt n,(x+\eta)\sqrt n]\}} |Z_{t_n+1}|^{-\beta}\Big{|}\mathcal{F}_{t_n}\right]\right]\cr
&= C_MC_4(2M)^{-\beta}\mathbb{E}\left[\frac{|Z_{t_n}|}{|Z_{t_n+1}|^{\beta}}\right]\nu_{t_n}\left([(x-\eta)\sqrt n,(x+\eta)\sqrt n]\right),
 \end{align}
where the last equality follows from the fact that the branching and  motion are independent. By Fatou's Lemma, for $n$ large enough,
\begin{align}\label{bbb}
\mathbb{E}\left[\frac{|Z_{t_n}|}{|Z_{t_n+1}|^{\beta}}\right]
=m^{-\beta}m^{-(\beta-1)t_n}\mathbb{E}\left[\frac{W_{t_n}}{(W_{t_{n+1}})^{\beta}}\right]
\geq 0.9\mathbb{E}\left[W^{1-\beta}\right] m^{-\beta}m^{-(\beta-1)t_n}.
\end{align}
Plugging (\ref{bbb}) into (\ref{sadfefw}), yields that there exists a constant $C'(\epsilon,M,\beta,h,x,\eta,\delta,p)$ such that for $n>C'(\epsilon,M,\beta,h,x,\eta,\delta,p)$,
\begin{align}
\mathbb{P}\Big(\bar{Z}_n(\sqrt nA)\geq p\Big)&\geq C_MC_4(2M)^{-\beta}0.9m^{-\beta}\mathbb{E}\left[W^{1-\beta}\right]m^{-(\beta-1)t_n}\nu_{t_n}\left([(x-\eta)\sqrt n,(x+\eta)\sqrt n]\right)\cr
&\geq C_MC_4(2M)^{-\beta}0.9m^{-\beta}\mathbb{E}\left[W^{1-\beta}\right] m^{-(\beta-1)h\sqrt n}e^{-(\Lambda(\frac{x-\eta}{h})+\epsilon)h\sqrt n},\nonumber
\end{align}
where the last inequality follows by Cram\'er's theorem. Hence, for every $\epsilon,~\eta$ small enough and $h>0$, the above yields that
 $$\liminf_{n\rightarrow\infty}\frac{1}{\sqrt n}\log\mathbb{P}\Big(\bar{Z}_n(\sqrt nA)\geq p\Big)\geq-\left\{h(\beta-1)\log m+\left(\Lambda\left(\frac{x-\eta}{h}\right)+\epsilon \right)h\right\}.$$
First let $\epsilon\rightarrow 0$, and then maximum  the lower bound with $h$, we finally obtain that
 $$\liminf_{n\rightarrow\infty}\frac{1}{\sqrt n}\log\mathbb{P}\Big(\bar{Z}_n(\sqrt nA)\geq p\Big)\geq-\inf_{h>0}\left\{h(\beta-1)\log m+\Lambda\left(\frac{I_A(p)}{h}\right)h\right\},$$
 which concludes the desired lower bound for the case of $\beta-1<\frac{-\log{p_1}}{\log m}$. For $\beta-1\geq \frac{-\log{p_1}}{\log m}$, \cite[Lemma 3.1]{ChenHe} gives
 $$
 \liminf_{n\rightarrow\infty}\frac{1}{\sqrt n}\log\mathbb{P}(\bar{Z}_n(\sqrt nA)\geq p)\geq -\inf_{h>0}\left\{h\log \frac{1}{p_1}+\Lambda\left(\frac{I_A(p)}{h}\right)h\right\}.
 $$

 \bigskip
 \textbf{ Upper bound}. By the definition of $I_A(p)$, for every $\eta\in(0,I_A(p))$, there exists $\delta>0$ such that
\begin{equation}\label{copyuper1}
\sup\limits_{|y|\leq I_A(p)-\eta}\nu(A-y)\leq p-\delta.
\end{equation}
 Set $B_n:=[(-I_A(p)+\eta)\sqrt n,(I_A(p)-\eta)\sqrt n]$, $\mathcal{M}_1:=\{\xi\in \mathcal{M}:\frac{\xi(B_n^c)}{|\xi|}\leq \frac{\delta}{2}\}$, $t_n:=\lfloor h\sqrt n\rfloor$ for some $h>0$. By \cite[(4.13)]{ChenHe},  for every $\xi\in\mathcal{M}_1$ and $n$ large enough, we have
$$
\frac{1}{|\xi|}\sum\limits_{z\in\xi}\nu_{n-{t_n}}(\sqrt n A-S_z)+\frac{\delta}{4}\leq \frac{1}{|\xi|}\sum\limits_{z\in\xi}\nu_{n-{t_n}}(\sqrt n A-S_z)\leq p.
$$
Hence, by Lemma \ref{louidorlem} and Markov inequality, for $n$ large enough,
\begin{align}\label{copyuper2}
&\mathbb{P}\left(\bar{Z}_{n}(\sqrt n A)\geq p\right)\cr
&\leq \mathbb{P}\left(\bar{Z}_{t_n}(B_n^c)\geq \frac{\delta}{2}\right)+\mathbb{P}\left(\bar{Z}_{t_n}(B_n^c)\leq \frac{\delta}{2},\bar{Z}_{n}(\sqrt n A)\geq p\right)\cr
&\leq \mathbb{E}\left[\frac{2}{\delta}\frac{Z_{t_n}(B_n^c)}{|Z_{t_n}|}\right]+\int_{\mathcal{M}_1}\mathbb{P}\left(\bar{Z}^\xi_{n-t_n}(\sqrt nA)\geq \frac{1}{|\xi|}\sum_{z\in\xi}\nu_{n-t_n}(\sqrt nA-S_z)+\frac{\delta}{4}\right)\mathbb{P}(Z_{t_n}\in d\xi)\cr
&\leq \mathbb{E}\left[\frac{2}{\delta}\frac{Z_{t_n}(B_n^c)}{|Z_{t_n}|}\right]+C_1\mathbb{E}\left[|Z_{t_n}|^{-(\beta-1)}\right].
\end{align}
For the first term  on the r.h.s above of (\ref{copyuper2}), define $\mathcal{G}_{t_n}:=\sigma(|Z_i|, 1\leq i\leq t_n)$. Since the branching and motion are independent, then for $n$ large enough,
\begin{equation}\label{werre}
\mathbb{E}\left[\frac{2}{\delta}\frac{Z_{t_n}(B_n^c)}{|Z_{t_n}|}\right]
=\mathbb{E}\left[\mathbb{E}\left[\frac{2}{\delta}\frac{Z_{t_n}(B_n^c)}{|Z_{t_n}|}\Big|\mathcal{G}_{t_n}\right]\right]
= \frac{4}{\delta}\nu_{t_n}\left([(I_A(p)-\eta)\sqrt n)\right)\leq\frac{4}{\delta} e^{-\left(\Lambda\left(\frac{I_A(p)-\eta}{h}\right)-\epsilon\right)h\sqrt n},
\end{equation}
where the last inequality follows by Cram\'er's theorem. For the second term  on the r.h.s of (\ref{copyuper2}), by Lemma \ref{Neyharmonic}, for $n$ large enough,
\begin{equation}\label{jkuy}
\mathbb{E}\left[|Z_{t_n}|^{-(\beta-1)}\right]
\leq\begin{cases}
2C_0 m^{-(\beta-1)h\sqrt n}, &\beta-1\geq \frac{-\log{p_1}}{\log m}; \cr
2C_0 p_1^{h\sqrt n}, &0<\beta-1<\frac{-\log{p_1}}{\log m}.
\end{cases}
\end{equation}
Plugging (\ref{werre}) and (\ref{jkuy}) into (\ref{copyuper2}), yields that: if $\beta-1\geq \frac{-\log{p_1}}{\log m}$, then
$$\mathbb{P}\left(\bar{Z}_{n}(\sqrt n A)\geq p\right)\leq \frac{4}{\delta}e^{-\left(\Lambda\left(\frac{I_A(p)-\eta}{h}\right)-\epsilon\right)h\sqrt n}+2C_0 C_1(p,\delta,h)m^{-(\beta-1)h\sqrt n};$$
if $0<\beta-1<\frac{\log \frac{1}{p_1}}{\log m}$, then
$$\mathbb{P}\left(\bar{Z}_{n}(\sqrt n A)\geq p\right)\leq \frac{4}{\delta}e^{-\left(\Lambda\left(\frac{I_A(p)-\eta}{h}\right)-\epsilon\right)h\sqrt n}+2C_0 C_1(p,\delta,h)p_1^{h\sqrt n}.$$
 So, the upper bound follows by optimizing $h$ on $(0,+\infty)$.
\end{proof}
\section{Proof of Theorem \ref{main5}.}\label{offspringweibull}
In this section, we consider the case that the offspring law has Weibull tail, i.e., $\mathbb{P}(|Z_1|>x)\sim l_1e^{-l x^{\beta}}$ as $x\rightarrow\infty$ for some constants $\beta\in(0,1)$ and $l_1,~l\in(0,+\infty)$. We assume that $I_A(p)<\infty$ and $I_A(\cdot)$ is continuous at $p$. Comparing with the case that the offspring law has exponential moment, the results and proofs have no change in the Schr\"{o}der case. However, in the B\"{o}ttcher case, things become different--the B\"{o}ttcher constant will appear in the rate function. So in this section, we further assume $p_1=0$. Moreover, unlike the Pareto case, the tail of step size matters for the decay scale of LDP. To show this, we will investigate two types of step size--the bounded step size and Weibull tail step size.
\subsection{Proof of (\ref{weibullbounded}).}
In this subsection, we assume $p_1=0$ and  $0<\text{ess}\sup X=L<\infty$. We are going to show that
$$\lim_{n\rightarrow\infty}\frac{1}{\sqrt{n}}\log\left[-\log \mathbb{P}(\bar{Z}_n(\sqrt nA)\geq p)\right]=\frac{I_A(p)}{L} \frac{\rho\beta}{\beta+\rho-\beta\rho}\log m.$$

\begin{proof}
\textbf{Lower bound}. By the continuity of $I_A(\cdot)$ at $p$, for every $\varepsilon>0$, there exists $\delta>0$ such that $|I_A(p+2\delta)-I_A(p)|<\varepsilon$. Furthermore, it is easy to see that there exists $x\in\mathbb{R}$ such that $\nu(A-x)=p+2\delta$. So, by the continuity of $\nu(A-\cdot)$, there exists some $\eta>0$ such that
$$\inf_{z\in\left(x,x+\frac{2\eta x}{L}\right)}\nu(A-z)\geq p+\delta,$$
where $(b,a):=(a,b)$, if $a<b$. Let $t_n:=\left\lceil\frac{|x|\sqrt n}{(L-\eta)}\right\rceil$ and
$$E:=\left\{\exists u\in Z_{t_n},~S_u\in\left(x\sqrt n,x\sqrt n+\frac{2\eta x\sqrt n}{L}\right),~|Z^{u}_{1}|>2M\sum_{v\neq u,v\in Z_{t_n}}|Z^{v}_{1}|\right\}.$$
Since $|Z_1|\sim$Weibull$(\beta)$, there exist constants $c_7>0$ and $l>0$ such that $\mathbb{P}(|Z_1|>y)>c_7e^{-ly^\beta}$ for all $y>0$. By similar arguments for the lower bound in Theorem \ref{main1}, there exist constants $C_M>0$ and $T(\eta,x)>0$ such that for $n$ large enough,
\begin{align}\label{efewfw}
\mathbb{P}\Big(\bar{Z}_n(\sqrt nA)\geq p\Big)&\geq C_M\mathbb{P}(E)\cr
&=C_M\mathbb{P}\left(S_{t_n}\in\left(x\sqrt n,x\sqrt n+\frac{2\eta x\sqrt n}{L}\right)\right)c_7\mathbb{E}\left[|Z_{t_n}|e^{-l(2M)^{\beta}|Z_{t_n+1}|^{\beta}}\right]\cr
&\geq C_Me^{-T(\eta,x)t_n}\mathbb{E}\left[e^{-l(2M)^{\beta}|Z_{t_n+1}|^{\beta}}\right],
\end{align}
where the last inequality follows by Cram\'er's theorem and the fact that $|x\sqrt n|\leq \left\lceil\frac{|x|\sqrt n}{(L-\eta)}\right\rceil L$.
Hence, by Lemma \ref{Weibullmoment}, we have
$$\limsup_{n\rightarrow\infty}\frac{1}{\sqrt n}\log\left[-\log \mathbb{P}(\bar{Z}_n(\sqrt nA)\geq p)\right]\leq \frac{|x|}{L-\eta}\frac{\rho\beta}{\beta+\rho-\beta\rho}\log m.$$
Finally, by letting $\eta\rightarrow0$ and $\varepsilon\rightarrow0$, we obtain the desired lower bound.

\bigskip
\textbf{Upper bound}. For any $\varepsilon\in(0,I_A(p))$, set $t_n:=\lfloor (I_A(p)-\varepsilon)\sqrt n/L\rfloor$. By Lemma \ref{Weibulllemma} and arguments from (2.30) to (2.33) in \cite{Louidor}, there exists $\delta>0$ such that for $n$ large enough,
\begin{align}
\mathbb{P}\Big(\bar{Z}_n(\sqrt nA)\geq p\Big)&\leq \int_{\mathcal{M}}\mathbb{P}\left(\bar{Z}^\xi_{n-t_n}(\sqrt nA)\geq \frac{1}{|\xi|}\sum_{z\in\xi}\nu_{n-t_n}(\sqrt nA-S_z)+\frac{\delta}{3}\right)\mathbb{P}(Z_{t_n}\in d\xi)\cr
 &\leq C_2\mathbb{E}(e^{-C_3|Z_{t_n}|^{\beta}}).\nonumber
 \end{align}
Thus, by Lemma \ref{Weibullmoment}, we have
$$\liminf_{n\rightarrow\infty}\frac{1}{\sqrt n}\log\left[-\log \mathbb{P}(\bar{Z}_n(\sqrt nA)\geq p)\right]\geq \frac{I_A(p)-\varepsilon}{L}\frac{\rho\beta}{\beta+\rho-\beta\rho}\log m.$$
At last, the desired upper bound follows by letting $\varepsilon\rightarrow 0$.
\end{proof}
\subsection{Proof of (\ref{weibullweibull}).}
In this subsection, we consider the case: the step size $X$ satisfies $\mathbb{P}(X>x)=\Theta(1)e^{-\lambda x^{\alpha}}$ as $x\rightarrow\infty$ for some $\alpha\in(0,\infty)$, $\lambda>0$. We are going to show: if $\alpha>1$, then
$$\lim\limits_{n\rightarrow\infty}\frac{(\log n)^{\alpha-1}}{n^{\alpha/2}}
\log\mathbb{P}(\bar{Z}_n(\sqrt nA)\geq p)= -\lambda {I_A(p)}^{\alpha}\left(\frac{2\beta\rho\log m}{\alpha(\beta+\rho-\beta\rho)}\right)^{\alpha-1}.$$
For $\alpha\in(0,1]$, using the same arguments as \cite[Sec 4.1.1]{ChenHe} and Lemma \ref{Weibulllemma}, we can easily obtain
 $$\lim\limits_{n\rightarrow\infty}\frac{1}{n^{\alpha/2}}
\log\mathbb{P}(\bar{Z}_n(\sqrt nA)\geq p)= -\lambda {I_A(p)}^{\alpha}.$$
So, we feel free to omit its proof here.
\begin{proof}
\textbf{Lower bound}. Fix $\epsilon>0$. There exist some $\eta,\delta>0$ and $|x|<I_A(p)+\epsilon$ such that
    $$\inf_{y\in(x-\eta,x+\eta)}\nu(A-y)\geq p+\delta.$$
    Let
    $$E:=\left\{\exists u\in Z_{t_n},~S_u\in\left((x-\eta)\sqrt n,(x+\eta)\sqrt n\right),~|Z^{u}_{1}|>2M\sum_{v\neq u,v\in Z_{t_n}}|Z^{v}_{1}|\right\},$$
    where $t_n:=\lfloor t\log n\rfloor$ for some $0<t<\frac{\alpha(\beta+\rho-\beta\rho)}{2\log(m+\epsilon)\beta\rho}$. Similar to (\ref{efewfw}), we have for $n$ large enough,
\begin{align}
\mathbb{P}\Big(\bar{Z}_n(\sqrt nA)\geq p\Big)&\geq C_M\mathbb{P}(E)\cr
&=C_M\mathbb{P}\left(S_{t_n}\in\left((x-\eta)\sqrt n,(x+\eta)\sqrt n\right)\right)c_7\mathbb{E}\left[|Z_{t_n}|e^{-l(2M)^{\beta}|Z_{t_n+1}|^{\beta}}\right]\cr
&\geq c_7C_M\exp\left\{-(\lambda+\epsilon)|x-\eta|^{\alpha}\frac{n^{\alpha/2}}{(t\log n)^{\alpha-1}}\right\}\mathbb{E}\left[e^{-l(2M)^{\beta}|Z_{t_n+1}|^{\beta}}\right]\cr
&\geq c_7C_M\exp\left\{-(\lambda+\epsilon)|x-\eta|^{\alpha}\frac{n^{\alpha/2}}{(t\log n)^{\alpha-1}}\right\}
\exp\left(-n^\frac{\beta\rho t\log (m+\epsilon)}{\beta+\rho-\beta\rho}\right)\cr
&\geq c_7C_M\exp\left\{-(\lambda+2\epsilon) |x-\eta|^{\alpha}\frac{n^{\alpha/2}}{(t\log n)^{\alpha-1}}\right\},\nonumber
\end{align}
where the second inequality follows from Lemma \ref{WeibullLDP}, the third inequality comes from Lemma \ref{Weibullmoment}, and the last inequality follows from the fact that $t<\frac{\alpha(\beta+\rho-\beta\rho)}{2\log(m+\epsilon)\beta\rho}$. Hence, for any $\epsilon>0$, some $\eta>0$, $|x|<I_A(p)+\epsilon$ and any $0<t<\frac{\alpha(\beta+\rho-\beta\rho)}{2\log(m+\epsilon)\beta\rho}$, we have
$$\liminf\limits_{n\rightarrow\infty}\frac{(\log n)^{\alpha-1}}{n^{\alpha/2}}
\log\mathbb{P}(\bar{Z}_n(\sqrt nA)\geq p)\geq -(\lambda+\epsilon)\frac{|x-\eta|^{\alpha}}{t^{\alpha-1}}.$$
Finally, by letting $\epsilon\rightarrow 0$, $t\rightarrow \frac{\alpha(\beta+\rho-\beta\rho)}{2\log m\beta\rho}$ gives the desired lower bound.

\bigskip
\textbf{Upper bound}. Set $t_n:= \lfloor t\log n\rfloor$ for $t>\frac{\alpha(\beta+\rho-\beta\rho)}{2\beta\rho\log(m-\epsilon)}$;  $B_n:=[(-I_A(p)+\eta)\sqrt n,(I_A(p)-\eta)\sqrt n]$. Using the arguments from (\ref{copyuper1}) to (\ref{copyuper2}), there exists some $\delta>0$ such that for $n$ large enough,
\begin{align}
&\mathbb{P}\left(\bar{Z}_{n}(\sqrt n A)\geq p\right)\cr
&\leq \mathbb{E}\left[\frac{2}{\delta}\frac{Z_{t_n}(B_n^c)}{|Z_{t_n}|}\right]+\int_{\mathcal{M}_1}\mathbb{P}\left(\bar{Z}^\xi_{n-t_n}(\sqrt nA)\geq \frac{1}{|\xi|}\sum_{z\in\xi}\nu_{n-t_n}(\sqrt nA-S_z)+\frac{\delta}{4}\right)\mathbb{P}(Z_{t_n}\in d\xi)\cr
&\leq \frac{2}{\delta}\mathbb{P}(S_{t_n}\in B^c_n)+C_2\mathbb{E}\left[e^{-C_3|Z_{t_n}|^{\beta}}\right],\cr
\end{align}
 where the last inequality follows from Lemma \ref{Weibulllemma}. As a consequence, by Lemma \ref{WeibullLDP} and Lemma \ref{Weibullmoment}, for $n$ large enough,
\begin{align}
\mathbb{P}\left(\bar{Z}_{n}(\sqrt n A)\geq p\right)&\leq  \exp\left\{-\lambda (I_A(p)-\eta-\epsilon)^{\alpha}\frac{n^{\alpha/2}}{(t\log n)^{\alpha-1}}\right\}+C_2\exp\left(-n^\frac{\beta\rho t\log (m-\epsilon)}{\beta+\rho-\beta\rho}\right)\cr
&\leq  \exp\left\{-(\lambda-2\epsilon)(I_A(p)-\eta)^{\alpha}\frac{n^{\alpha/2}}{(t\log n)^{\alpha-1}}\right\},\nonumber
\end{align}
where  the last inequality follows by $t>\frac{\alpha(\beta+\rho-\beta\rho)}{2\beta\rho\log(m-\epsilon)}$. So, for any $\epsilon,~\eta>0$ small enough and $t>\frac{\alpha(\beta+\rho-\beta\rho)}{2\beta\rho\log(m-\epsilon)}$, we have
$$\limsup\limits_{n\rightarrow\infty}\frac{(\log n)^{\alpha-1}}{n^{\alpha/2}}
\log\mathbb{P}(\bar{Z}_n(\sqrt nA)\geq p)\geq -(\lambda-2\epsilon)\frac{(I_A(p)-\eta)^{\alpha}}{t^{\alpha-1}},$$
 which implies the upper bound, by letting $\epsilon,~\eta\rightarrow 0$ and $t\rightarrow \frac{\alpha(\beta+\rho-\beta\rho)}{2\beta\rho\log m}$.
\end{proof}
  \section{Proof of Theorem \ref{main3}.}\label{Paretostepsize}
  In this section, we assume $\mathbb{E}[|Z_1|^{\beta}]<\infty$ for some $\beta>1$ and $I_A(\cdot)$  is continuous at $p$ for some $p\in(\nu(A),1-\nu(A))$. Here, we consider the step size has Pareto tail, i.e. $\mathbb{P}(X>x)\sim \kappa x^{-\alpha}$ as $x\rightarrow\infty$ for some constants $\kappa>0$ and $\alpha>2$. We are going to show that: if $b<B$, then
  $$\lim\limits_{n\rightarrow\infty}\frac{1}{\log n}\log\mathbb{P}(\bar{Z}_n(\sqrt nA)\geq p)=-\frac{\alpha}{2}.$$
Furthermore, if $b=B$ and $A$ is an unbounded set, the above still holds, provided $0<p-\nu(A)<(1-\nu(A))/b$.
\begin{proof}
\textbf{Lower bound}. By the continuity of $I_A(\cdot)$ at $p$, for every $\epsilon>0$, there exist $\eta,~\delta>0$ such that for some $|x|\leq I_A(p)+\epsilon$,
$$\inf\limits_{y\in[x-\eta,x+\eta]}\nu(A-y)\geq p+\delta.$$
 Without loss of generality, we write $A=\sum^h_{j=1}(a_j,b_j]$. Set $$A(x,\eta):=\sum^h_{j=1}\left(a_j-(x-\eta),b_j-(x+\eta)\right),~A(\epsilon):=\sum^h_{j=1}\left(a_j+\epsilon,b_j-\epsilon\right).$$
Obviously, we can choose $\epsilon,~\eta$ small enough such that
$$\nu(A(x,\eta))>p+\frac{\delta}{2},~\nu(A(\epsilon))>\nu(A)-\frac{\delta}{2}.$$
Set
 $$\mathcal{E}:=\left\{\xi\in\mathcal{M}: \xi=\sum^b_{i=1}\delta_{{x_i}\sqrt{n}},~ \text{where}~x_1\in(x-\eta,x+\eta),~x_i\in(-\epsilon,\epsilon),~i=2,...,b\right\}.$$
By Markov property,
$$\mathbb{P}(\bar{Z}_{n}(\sqrt n A)\geq p)\geq\int_\mathcal{E} \mathbb{P}\left(\bar{Z}^{\xi}_{n-1}(\sqrt n A)\geq p\right)\mathbb{P}(Z_1\in d\xi).$$
For every $\xi\in \mathcal{E}$, it is easy to see that
$$A-x_1\supset A(x,\eta),~A-x_i\supset A(\epsilon),~i=2,...,b.$$
Hence,
\begin{align}\label{erewrer}
\mathbb{P}(\bar{Z}_{n}(\sqrt n A)\geq p)&\geq\int_\mathcal{E}\mathbb{P}\left(\bar{Z}^{\xi}_{n-1}(\sqrt nA)\geq p\right)\mathbb{P}(Z_1\in d\xi)\cr
&\geq\int_\mathcal{E} \mathbb{P}\left(\frac{Z^{1}_{n-1}(\sqrt n A(x,\eta))+\sum^b_{i=2}Z^{i}_{n-1}(\sqrt n A(\epsilon))}{\sum^b_{i=1}|Z^{i}_{n-1}|}\geq p\right)\mathbb{P}(Z_1\in d\xi)\nonumber\\
&=\mathbb{P}\left(\frac{Z^{1}_{n-1}(\sqrt n A(x,\eta))+\sum^b_{i=2}Z^{i}_{n-1}(\sqrt n A(\epsilon))}{\sum^b_{i=1}|Z^{i}_{n-1}|}\geq p\right)\mathbb{P}(Z_1\in \mathcal{E}),
\end{align}
where $Z^{i}_{n-1}(\cdot),~1\leq i\leq b$ are i.i.d copies of $Z_{n-1}(\cdot)$.
Since $\bar{Z}_{n}(\sqrt n A)\rightarrow \nu(A)$  and $|Z_n|m^{-n}\rightarrow W$ almost surely, we have
\begin{equation}\label{a.s.shoul}
\frac{Z^{1}_{n-1}(\sqrt n A(x,\eta))+\sum^b_{i=2}Z^{i}_{n-1}(\sqrt n A(\epsilon))}{\sum^b_{i=1}|Z^{i}_{n-1}|}\rightarrow
\frac{\sum^b_{i=2}\nu(A(\epsilon))W_i+\nu(A(x,\eta))W_1}{\sum^b_{i=1}W_i},~\text{as}~n\rightarrow\infty,~\mathbb{P}-a.s.,
\end{equation}
where $W_i,~1\leq i\leq b$ are i.i.d copies of $W$. If $b<B$, then $W_i$ has a continuous density on $(0,+\infty)$ (see \cite[Chapter II, Lemma 2]{Athreya}). In this case, by the dominated convergence theorem and (\ref{a.s.shoul}),
\begin{equation}\label{dfrfw}
\lim\limits_{n\rightarrow\infty}\mathbb{P}\left(\frac{Z^{1}_{n-1}(\sqrt n A(x,\eta))+\sum^b_{i=2}Z^{i}_{n-1}(\sqrt n A(\epsilon))}{\sum^b_{i=1}|Z^{i}_{n-1}|}\geq p\right)=:C(\epsilon,\eta,x)>0.
\end{equation}
Since $\mathbb{P}(X>x)\sim\kappa x^{-\alpha}$, there exists a constant $C(\kappa,x,\epsilon,\eta)=: C_7>0$ such that for $n$ large enough,
\begin{equation}\label{gfhg}
\mathbb{P}(Z_1\in\mathcal{E})\geq C_7p_bn^{-\alpha/2}.
\end{equation}
Plugging (\ref{dfrfw}) and (\ref{gfhg}) into (\ref{erewrer}), yields that for $n$ large enough,
\begin{align}
\mathbb{P}(\bar{Z}_{n}(\sqrt n A)\geq p)\geq 0.9C(\epsilon,\eta,x)C_7p_bn^{-\alpha/2},\nonumber
\end{align}
which implies the desired lower bound if $b<B$.\par
If $b=B$ and $A$ is unbounded, without loss of generality, we assume $A=(a,+\infty)$. Set

 $$\mathcal{E'}:=\left\{\xi\in\mathcal{M}: \xi=\sum^b_{i=1}\delta_{{x_i}\sqrt{n}},~ \text{where}~x_1\in[t\sqrt n,+\infty),~x_i\in(-\epsilon,\epsilon),~i=2,...,b\right\}.$$
where $t$ is some positive constant. Using similar arguments as above, one can obtain that
$$
\mathbb{P}(\bar{Z}_{n}(\sqrt n A)\geq p)\geq\mathbb{P}\left(\frac{Z^{1}_{n-1}(\sqrt n(a-t,+\infty))+\sum^b_{i=2}Z^{i}_{n-1}(\sqrt n A(\epsilon))}{\sum^b_{i=1}Z^{i}_{n-1}}\geq p\right)\mathbb{P}(Z_1\in \mathcal{E'}).
$$
For any $p\in\left(\nu(A),\nu(A)+\frac{1-\nu(A)}{b}\right)$, there exists $\delta'>0$ such that
 \begin{equation} \label{asshoul}
 p<\frac{(b-1)(\nu(A)-\delta')+1-\delta'}{b}.
 \end{equation}
If we choose $\epsilon$ small enough and $t$ large enough, then
$$\nu(A(\epsilon))\geq \nu(A)-\delta',~\nu((a-t,+\infty))>1-\delta'.$$
This, together with (\ref{a.s.shoul})and (\ref{asshoul}), yields
$$\frac{Z^{1}_{n-1}(\sqrt n(a-t,+\infty))+\sum^b_{i=2}Z^{i}_{n-1}(\sqrt n A(\epsilon))}{\sum^b_{i=1}|Z^{i}_{n-1}|}\rightarrow
\frac{(b-1)\nu(A(\epsilon))+\nu((a-t,+\infty))}{b}>p,~\mathbb{P}-a.s.,$$
which implies the desired lower bound.

\bigskip
\textbf{Upper bound}. Set $t_n:=\lfloor c\log n\rfloor$ for some $c>0$. By copying the arguments from (\ref{copyuper1}) to (\ref{copyuper2}), there exist constants $\delta,~\eta >0$ such that for $n$ large enough,
\begin{align}\label{imB}
&\mathbb{P}\left(\bar{Z}_{n}(\sqrt n A)\geq p\right)\leq \frac{4}{\delta}\mathbb{P}\left(S_{t_n}\geq(I_A(p)-\eta)\sqrt n\right)+ C_1\mathbb{E}\left[|Z_{t_n}|^{-(\beta-1)}\right],\nonumber
\end{align}
where the last inequality follows from Lemma \ref{louidorlem}. For the first term of r.h.s above, by (\ref{Nagaevinequailty}), there exist a constant $C_8>0$ depending on $c,~\kappa,~\alpha,~\eta$ such that for $n$ large enough,
\begin{equation}\label{ii}
\frac{4}{\delta}\mathbb{P}\left(S_{t_n}\geq(I_A(p)-\eta)\sqrt n\right)
\leq C_8 n^{-\frac{\alpha}{2}}\log n,
\end{equation}
For the second term, by Lemma \ref{Neyharmonic}, for $n$ large enough,
\begin{align}
\mathbb{E}\left[|Z_{t_n}|^{-(\beta-1)}\right]\leq 2C_0\left(m^{-t_n(\beta-1)}+p_1^{t_n}\right).
\end{align}
Hence, for $n$ large enough,
$$\mathbb{P}\left(\bar{Z}_{n}(\sqrt n A)\geq p\right)\leq C_8n^{-\frac{\alpha}{2}}\log n+2C_0\left(m^{-t_n(\beta-1)}+p_1^{t_n}\right).$$
Since $t_n=\lfloor c\log n\rfloor$, if we choose $c>\frac{\alpha}{2(\beta-1)\log m}\vee\frac{\alpha}{-2\log p_1}$, then for large $n$,
$$\mathbb{P}\left(\bar{Z}_{n}(\sqrt n A)\geq p\right)\leq 2C_8n^{-\frac{\alpha}{2}}\log n.$$
Thus, we have
$$\limsup\limits_{n\rightarrow\infty}\frac{1}{\log n}\log\mathbb{P}(\bar{Z}_n(\sqrt nA)\geq p)\leq-\frac{\alpha}{2}.$$
\end{proof}
\section{Proof of Theorem \ref{main4}.}
In this subsection, we assume $I_A(p)=\infty$, $J_A(p)$ is continuous at $p$ and $\mathbb{E}[X^2]<\infty$.  We are going to prove the following:\\
if $|Z_1|\sim$Pareto$(\beta)$ with some $\beta>1$, then
$$
\lim\limits_{n\rightarrow\infty}\frac{1}{n}\log\mathbb{P}(\bar{Z}_n(\sqrt nA)\geq p)
=\begin{cases}
-J_A(p)(\beta-1)\log m, &0<\beta-1<\frac{-\log p_1}{\log m}; \cr
J_A(p)\log{p_1}, &\beta-1\geq \frac{-\log p_1}{\log m},
\end{cases}
$$
if $|Z_1|\sim$Weibull$(\beta)$ with some $\beta\in(0,1)$, then
$$
\lim\limits_{n\rightarrow\infty}\frac{1}{n}\log\mathbb{P}(\bar{Z}_n(\sqrt nA)\geq p)=J_A(p)\log{p_1},~\text{for}~p_1>0,
$$
and
$$\lim_{n\rightarrow\infty}\frac{1}{n}\log\left[-\log \mathbb{P}(\bar{Z}_n(\sqrt nA)\geq p)\right]=J_A(p) \frac{\rho\beta}{\beta+\rho-\beta\rho}\log m,~\text{for}~p_1=0.
$$
\begin{proof}
\textbf{Lower bound}.  We first consider that $|Z_1|$ has  Pareto tail. If $\beta-1\geq \frac{-\log{p_1}}{\log m}$, the  lower bound is exactly the same as \cite[Lemma 3.4]{ChenHe}. Now we shall consider the case of $0<\beta-1<\frac{-\log{p_1}}{\log m}$. By exactly the same way as \cite[Lemma 3.4]{ChenHe}, for any $\epsilon>0$, there exist $r\in(0,1),~x\in\mathbb{R},~\delta>0$ and $\eta>0$ such that
 $$|r-J_A(p)|<\epsilon~\text{and}~\nu\left(\bigcap\limits_{y\in[x-\eta,x+\eta]}\frac{A-y}{\sqrt{1-r}}\right)\geq p+\delta.$$
Set $t_n:=\lfloor rn\rfloor$, and for large $M$ set
 $$E:=\left\{\exists u\in Z_{t_n},~S_u\in[(x-\eta)\sqrt n,(x+\eta)\sqrt n],~|Z^{u}_{1}|>2M\sum_{v\neq u,v\in Z_{t_n}}|Z^{v}_{1}|\right\}.$$
Using similar arguments for the lower bound in Theorem \ref{main1}, there exists a constant $C_M>0$ such that for $n$ large enough,
\begin{align}\label{dferf}
\mathbb{P}\Big(\bar{Z}_n(\sqrt nA)\geq p\Big)
&\geq C_M\mathbb{P}\left(E\right)\cr
&\geq C_M 0.9C_0m^{-t_n(\beta-1)}\nu_{t_n}\left([(x-\eta)\sqrt n,(x+\eta)\sqrt n]\right)\cr
&\geq  0.9^2C_MC_0C_9m^{-t_n(\beta-1)},
\end{align}
where the last inequality holds since
$$\lim\limits_{n\rightarrow\infty}\nu_{t_n}\left([(x-\eta)\sqrt n,(x+\eta)\sqrt n]\right)=C(r,x,\eta)=:C_9>0.$$
Taking limits in (\ref{dferf}) yields that
 $$\liminf\limits_{n\rightarrow\infty}\frac{1}{n}\log\mathbb{P}(\bar{Z}_n(\sqrt nA)\geq p)\geq-(J_A(p)+\epsilon)(\beta-1)\log m.$$
Then, the desired lower bound follows by letting $\epsilon \rightarrow 0$.

\bigskip
\textbf{Upper bound}. For  $\epsilon \in(0,J_A(p))$, set $t_n:=\lfloor(J_A(p)-\epsilon)n\rfloor$, by the arguments of \cite[Lemma 3.5]{ChenHe}, there exists $\delta>0$ such that for $n$ large enough,
$$\mathbb{P}(\bar{Z}_n(\sqrt nA)\geq p)\leq\int_{\mathcal{M}}\mathbb{P}\left(\bar{Z}^\xi_{n-t_n}(\sqrt nA)\geq \frac{1}{|\xi|}\sum_{y\in\xi}\nu_{n-t_n}(\sqrt nA-S_y)+\frac{\delta}{2}\right)\mathbb{P}(Z_{t_n}\in d\xi).$$
Hence, by Lemma \ref{louidorlem}, there exists $C_1>0$ such that for large $n$,
\begin{align}
\mathbb{P}(\bar{Z}_n(\sqrt nA)\geq p)\leq C_1\mathbb{E}\left[|Z_{t_n}|^{-(\beta-1)}\right].\nonumber
\end{align}
As a consequence, the upper bound follows by Lemma \ref{Neyharmonic}.\par
If $|Z_1|$ has Weibull tail, using similar arguments as above, together with Lemma \ref{Weibullmoment} and Lemma \ref{Weibulllemma}, we can get the desired results.
\end{proof}
\textbf{Acknowledgements}
I would like to thank my supervisor Hui He for his constant help while working
on this subject, all the useful discussions and advices.

\vspace{0.2cm}
Shuxiong Zhang\\
School of Mathematical Sciences, Beijing Normal University, Beijing 100875, People's Republic of China.\\
E-mail: shuxiong.zhang@qq.com.
\end{document}